\documentclass[a4paper]{article}
\usepackage[margin=1.2in]{geometry}
\usepackage{amsthm}
\usepackage{amssymb}
\usepackage{amsfonts}
\usepackage{subfigure}
\usepackage{amsmath,mathrsfs,bm}
\usepackage{array}
\usepackage{color,soul}
\usepackage{overpic}
\usepackage{tikz-cd}
\usepackage{ctable}
\usepackage{mathtools}
\usepackage{multirow}
\usepackage{authblk}

\newtheorem{theorem}{Theorem}[section]

\newtheorem{definition}[theorem]{Definition}
\newtheorem{remark}[theorem]{Remark}

\newcommand{\curl}{\bm{\mathrm{curl}}}

\numberwithin{equation}{section}

\begin{document}
\title{High accuracy analysis of a nonconforming discrete Stokes complex over rectangular meshes\thanks{This project is supported by NNSFC (Nos.~61733002, 61572096, 61432003, 61720106005, 61502107)
and ``the Fundamental Research Funds for the Central Universities''. }}

\author[a]{Xinchen Zhou\thanks{Corresponding author: dasazxc@gmail.com}}
\author[b]{Zhaoliang Meng}
\author[c]{Xin Fan}
\author[b,c]{Zhongxuan Luo}

\affil[a]{\it \small Faculty of Electronic Information and Electrical Engineering, Dalian University of Technology, Dalian 116024, China}
\affil[b]{\it \small School of Mathematical Sciences, Dalian
University of Technology, Dalian, 116024, China}
\affil[c]{\it School of Software, Dalian
University of Technology, Dalian, 116620, China}

\maketitle
\begin{abstract}
This work is devoted to the high accuracy analysis of a discrete Stokes complex
over rectangular meshes with a simple structure.
The 0-form in the complex is a non $C^0$ nonconforming element space for biharmonic problems.
This plate element contains only 12 degrees of freedom (DoFs) over a rectangular cell
with a zeroth order weak continuity for the normal derivative,
therefore only the lowest convergence order can be obtained by a standard consistency error analysis.
Nevertheless, we prove that, if the rectangular mesh is uniform,
an $O(h^2)$ convergence rate in discrete $H^2$-norm will be achieved.
Moreover, based on a duality argument,
it has an $O(h^3)$ convergence order in discrete $H^1$-norm if the solution region is convex.
The 1-form and 2-form constitute a divergence-free pair for incompressible flow.
We also show its higher accuracy than that derived from a usual error estimate under uniform partitions,
which explains the phenomenon observed in our previous work \cite{Zhou2018}.
Numerical tests verify our theoretical results.
\\[6pt]
\end{abstract}

\section{Introduction}

It is generally not an easy task to construct finite elements for biharmonic problems with high accuracy but a simple structure.
Although there are several conforming elements achieving high order convergence rate,
such as the Argyris element \cite{Argyris1968},
the Bogner-Fox-Schmit (BFS) element \cite{Bogner1965}, the Hsieh-Clough-Tocher (HCT) element \cite{Clough1965}
and the Fraijes de Veubeke-Sander (FVS) element \cite{Ciavaldini1974, Verbeke1968},
they suffer from the drawbacks caused by the strong $C^1$ continuity requirement.
For example, many degrees of freedom (DoFs) have to be utilized,
the degrees of the polynomials in shape function spaces must be very high,
or the structures of shape function spaces are too complicated.
An alternative way is to adopt nonconforming elements.
A successful construction is the Adini element on rectangular meshes,
whose DoFs are of the nodal type at the vertices of a mesh, resulting in fewer total DoFs.
In \cite{Lascaux1975}, it was shown that a second order convergence rate can be guaranteed in energy norm if the rectangular cells are of the same size, which is based on the interior symmetry of a rectangle.
Thereafter, Luo and Lin \cite{Luo2004} pointed out that the result above is still valid without such a uniformity assumption.
This fact was recently extended to higher-dimensional cases in \cite{Hu2016}.
However, one can only expect the same convergence order in discrete $H^1$- and $L^2$-norms due to the lower bound estimate given by Hu et al.~\cite{Hu2013, Hu2016}.
In recent years, the design for nonconforming elements draw rapidly increasing attention.
The bubble function technique has become a standard tool to achieve a high convergence order
if the desired element is $C^0$-continuous \cite{Gao2011,Wang2012,Chen2012,Chen2013}.
There are also some completely nonconforming constructions utilizing a high order inter-element continuity \cite{Zhang2018}.
In addition, we remark that the purpose of high accuracy can also be achieved by nonstandard methods such as the double set parameter (DSP) method \cite{Shi2000} or superconvergence analysis and postprocessing for low order elements \cite{Mao2009,Hu2016b}.

Biharmonic elements are often used to design divergence-free Stokes elements with discrete Stokes complexes.
A divergence-free Stokes element is usually preferred than a non divergence-free one,
as the former benefits from many advantages, say,
the conservation law is preserved numerically for incompressible flow,
the discrete scheme seems more robust and accurate with respect to the time discretization.
A comprehensive review of this topic can be found in \cite{John2016}.
Conforming examples include the Stokes elements and the discrete Stokes complexes
derived from the Argyris element \cite{Falk2013}, the BFS element \cite{Neilan2016},
the HCT element \cite{Christiansen2018} and the FVS element \cite{Neilan2018}, etc.
High order nonconforming complexes such as in \cite{Guzman2012} and \cite{Zhang2018} are also successful constructions.
Recently, the de Rham complex from the Adini element was also shown by Gillette et al.~\cite{Gillette2018} as a member of a nonstandard family.

In 1996, in order to construct a family of biharmonic elements by the DSP method,
Chen and Shi \cite{Chen1996} designed an intermediate rectangular element of 12 local DoFs,
which are the four point values, the four edge integrals of the shape functions
and the edge integrals of their normal derivatives.
The shape function space is selected as $P_3$ plus a space spanned by two monomials of degree four.
It was then used as a non $C^0$ nonconforming element for fourth order elliptic singular perturbation problems
with a uniform convergence property \cite{Chen2005}.
Very recently, we extended this method to general convex quadrilateral meshes
with some slight modifications to the shape function space,
and induced a divergence-free Stokes element for the Brinkman problem \cite{Zhou2018}.
The associated discrete Stokes complex was also provided.
Based on a standard consistency error estimate,
only the lowest order convergence rates of this complex can be derived.
Surprisingly, in our numerical tests,
we found that the velocity has an $O(h^2)$ approximation order in energy norm if the mesh consists of uniform rectangles,
one order higher than that given by the theoretical analysis therein.
Moreover, as far as we are aware,
although there have been many researches on the DSP method motivated by \cite{Chen1996},
the high accuracy analysis for the 12-DoF intermediate rectangular element introduced in \cite{Chen1996} is still unknown.

The aim of this work is to give a high accuracy analysis of the intermediate rectangular element given in \cite{Chen1996}
and the induced discrete Stokes complex.
We prove, by the special properties of the shape function space, that if the rectangular mesh is uniform,
this plate element has actually an $O(h^2)$ convergence rate in discrete $H^2$-norm
if it works for the biharmonic problem.
In addition, with the aid of an auxiliary biharmonic element and through a duality argument,
an $O(h^3)$ convergence order in discrete $H^1$-norm can be achieved,
provided that the solution region is convex.
In comparison with the well-known Adini element,
the numbers of local DoFs are the same.
Although there are more global DoFs,
this element space is highly nonconforming with many edge-oriented basis functions,
enjoying a low brand width when the stiffness matrix is accumulated.
This might be more convenient in implementation, especially for parallel computing.
Moreover, the convergence rate in discrete $H^1$-norm is one order faster than the Adini counterpart.
From this element, a nonconforming discrete complex is presented using the strategy in \cite{Zhou2018}.
The 1-form and 2-form constitute a divergence-free Stokes pair for incompressible flow.
Under uniform partitions, we also show its higher accuracy than that derived from a usual error estimate,
in both energy norm and $L^2$-norm for the velocity.
The convergence order of the pressure can be improved by a simple postprocessing.
This complex has the same key features as that provided in our previous work \cite{Zhou2018} over uniform rectangular meshes,
and so a rigorous explanation for the phenomenon observed in \cite{Zhou2018} is obtained.

The rest of this work is arranged as follows.
In Section \ref{s: stokes complex},
we give an overview of the investigated discrete Stokes complex.
Section \ref{s: scalar} provides the high accuracy analysis of the 0-form for biharmonic problems.
The error estimate for the Stokes pair determined by the 1-form and 2-form will be presented in Section \ref{s: vector}.
Finally, numerical examples are given in Section \ref{s: numer ex} to verify our theoretical analysis.

Standard notations in Sobolev spaces are used throughout this work.
For a domain $D\subset\mathbb{R}^2$,
$\bm{n}$ and $\bm{t}$ will be the unit outward normal and tangent vectors on $\partial D$, respectively.
The notation $P_k(D)$ denotes the usual polynomial space over $D$ of degree no more than $k$.
The norms and semi-norms of order $m$ in the Sobolev spaces $H^m(D)$
are indicated by $\|\cdot\|_{m,D}$ and $|\cdot|_{m,D}$, respectively.
The space $H_0^m(D)$ is the closure in $H^m(D)$ of $C_0^{\infty}(D)$.
The relation between $\bm{H}_0(\mathrm{div};\Omega)$ and $\bm{H}(\mathrm{div};\Omega)$ is in a similar manner.
We also adopt the convention that $L^2(D):=H^0(D)$,
where the inner-product is denoted by $(\cdot,\cdot)_D$.
The functions in its subspace $L_0^2(D)$ are of zero integral.
These notations of norms, semi-norms and inner-products also work for vector- and matrix-valued Sobolev spaces,
where the subscript $\Omega$ will be omitted if the domain $D=\Omega$.
Moreover, $C$ is a positive constant independent of the mesh size $h$ and may be different in different places.

\section{Preliminaries: A nonconforming discrete Stokes complex}
\label{s: stokes complex}

Let $\Omega\subset\mathbb{R}^2$ be a polygonal domain and $\partial\Omega$ be its boundary.
We assume that $\Omega$ can be uniformly partitioned into rectangular cells,
denoted by $\mathcal{T}_h$, with $h$ being the length of the diagonal of each rectangle.
For a cell $K\in\mathcal{T}_h$, the vertices are designated as $V_{i,K}$, $i=1,2,3,4$.
We select the coordinate system fulfilling that $V_{1,K}=(x'_K,y'_K)$, $V_{2,K}=(x''_K,y'_K)$,
$V_{3,K}=(x''_K,y''_K)$, $V_{4,K}=(x'_K,y''_K)$.
The edges $E_{i,K}=V_{i,K}V_{i+1,K}$ are parallel to the coordinate axes, whose lengths are naturally given by
\[
h_x=|E_{1,K}|=|E_{3,K}|=x''_K-x'_K,~h_y=|E_{2,K}|=|E_{4,K}|=y''_K-y'_K.
\]
Note that here $i$ is taken modulo four.
We can now give the following two elements over $K$.
See Figure \ref{fig: elements} as an example.

\begin{definition}
(See also \cite{Chen1996}) The element $(K,W_K,T_K)$ is defined through
\begin{itemize}
\setlength{\itemsep}{-\itemsep}
\item $K\in\mathcal{T}_h$ is a rectangle;
\item $W_K=P_3(K)\oplus\mathrm{span}\{x^4,y^4\}$ is the shape function space;
\item $T_K=\{\tau_{i,K},~i=1,2,\ldots,12\}$ is the DoF set where
\[
\begin{aligned}
\tau_{i,K}(w)&=w(V_{i,K}),~\tau_{i+4,K}(w)=\int_{E_{i,K}}w\,\mathrm{d}s,\\
\tau_{i+8,K}(w)&=\int_{E_{i,K}}\frac{\partial w}{\partial\bm{n}}\,\mathrm{d}s,
~\forall w\in W_K,~i=1,2,3,4.
\end{aligned}
\]
\end{itemize}
\end{definition}

\begin{definition}
The element $(K,\bm{V}_K,\Sigma_K)$ is defined through
\begin{itemize}
\setlength{\itemsep}{-\itemsep}
\item $K\in\mathcal{T}_h$ is a rectangle;
\item $\bm{V}_K$ is the shape function space, where
\[
\bm{V}_K=[P_1(K)]^2\oplus\mathrm{span}\{\curl\,x^3,\curl\,x^2y,\curl\,xy^2,\curl\,y^3,\curl\,x^4,\curl\,y^4\};
\]
\item $\Sigma_K=\{\sigma_{i,K},~i=1,2,\ldots,12\}$ is the DoF set:
\[
\begin{aligned}
\sigma_{i,K}(\bm{v})&=\int_{E_{i,K}}\bm{v}\cdot\bm{n}\,\mathrm{d}s,
~\sigma_{i+4,K}(\bm{v})=\int_{E_{i,K}}\bm{v}\cdot\bm{n}\xi_{i,K}\,\mathrm{d}s,\\
\sigma_{i+8,K}(\bm{v})&=\int_{E_{i,K}}\bm{v}\cdot\bm{t}\,\mathrm{d}s,
~\forall \bm{v}\in \bm{V}_K,~i=1,2,3,4,
\end{aligned}
\]
where $\xi_{i,K}\in P_1(K)$ are monomials in variable $x$ or $y$ such that
\[
\begin{aligned}
\xi_{1,K}&=\xi_{3,K},~\xi_{1,K}(V_{2,K})=\xi_{1,K}(V_{3,K})=-\xi_{1,K}(V_{1,K})=-\xi_{1,K}(V_{4,K})=1,\\
\xi_{2,K}&=\xi_{4,K},~\xi_{2,K}(V_{3,K})=\xi_{2,K}(V_{4,K})=-\xi_{2,K}(V_{1,K})=-\xi_{2,K}(V_{2,K})=1.
\end{aligned}
\]
\end{itemize}
\end{definition}

\begin{figure}[!htb]
\centering
\subfigure[DoFs of the element $(K,W_K,T_K)$] {
\label{fig: element 1}
\begin{overpic}[scale=0.4]{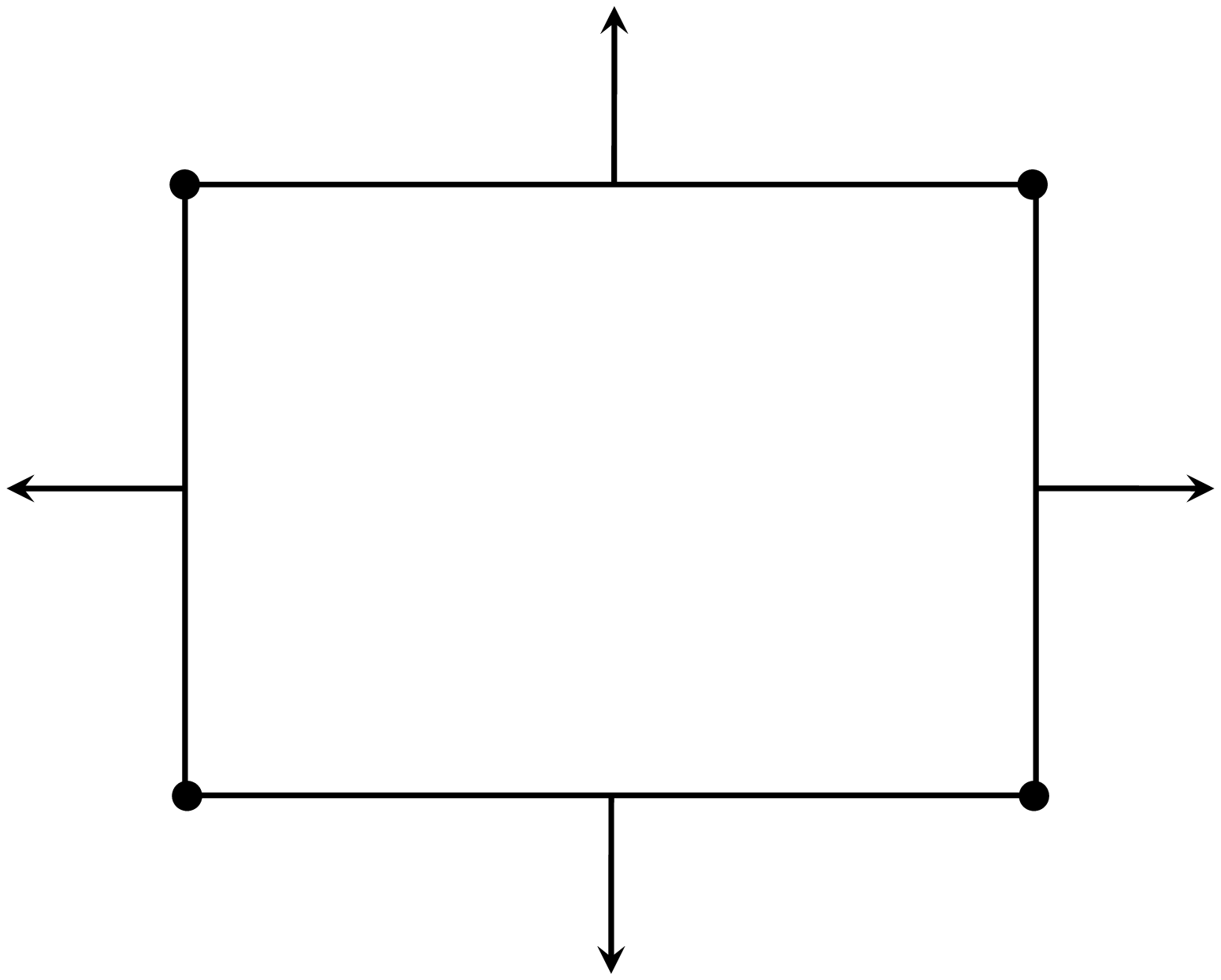}
\put(48,39){$K$} \put(8,10){$V_{1,K}$} \put(83,10){$V_{2,K}$} \put(83,67){$V_{3,K}$} \put(8,68){$V_{4,K}$}
\put(40,15){$\int$} \put(57,63){$\int$} \put(84,27){$\int$} \put(12,50){$\int$}
\end{overpic}
}
\subfigure[DoFs of the element $(K,\bm{V}_K,\Sigma_K)$] {
\label{fig: element 2}
\begin{overpic}[scale=0.4]{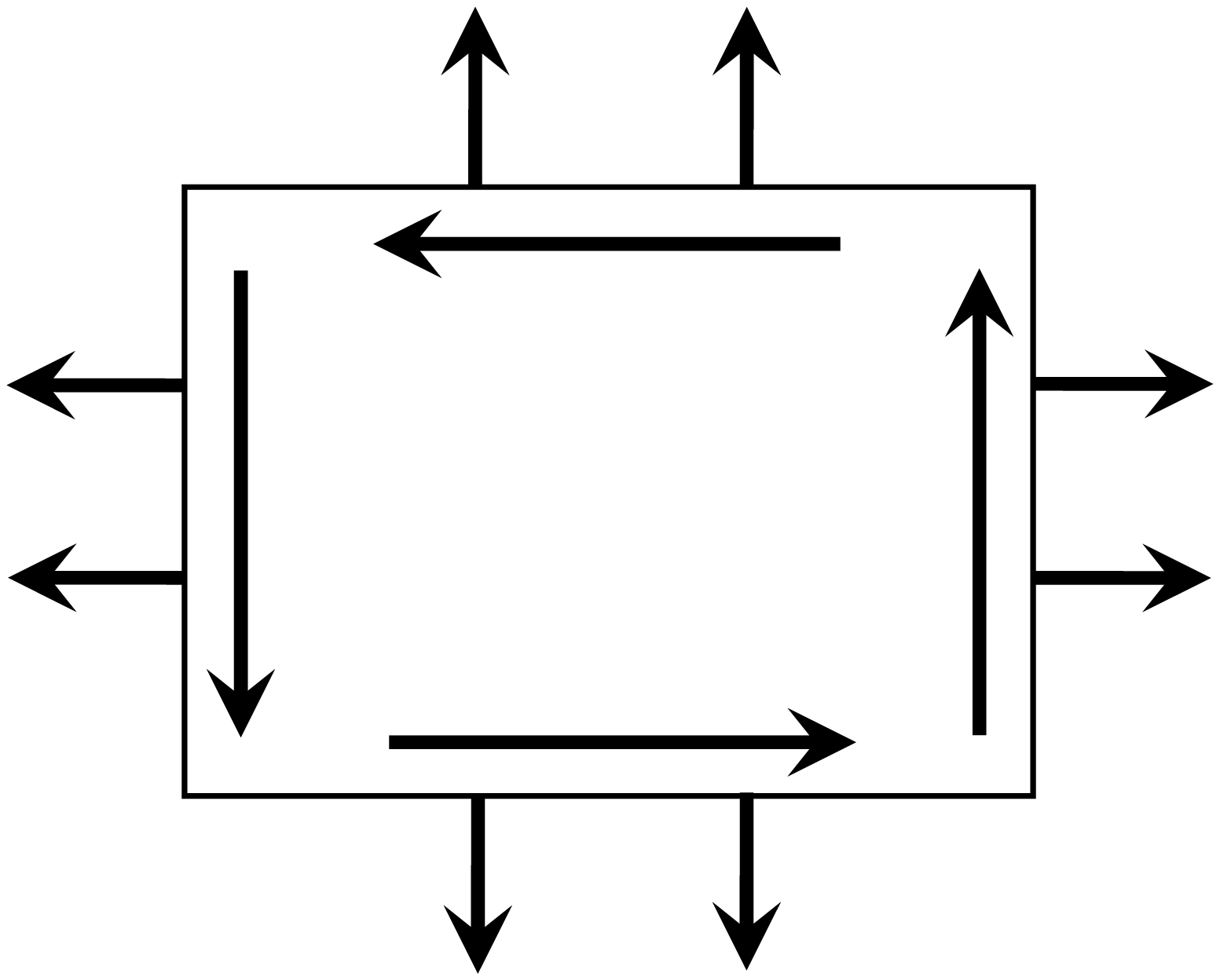}
\put(48,39){$K$}
\end{overpic}
}
\caption{DoFs of the two elements.\label{fig: elements}}
\end{figure}

\begin{theorem}
These two elements are well-defined.
\end{theorem}
\begin{proof}
The unisolvency of $W_K$ with respect to $T_K$ can be found in \cite{Chen1996}.
The assertion for the second element can be directly obtained by the technique provided in Theorem 2.5 in \cite{Zhou2018}
using the first element.
\end{proof}

For a given partition $\mathcal{T}_h$,
the sets of all vertices, interior vertices, boundary vertices,
edges, interior edges and boundary edges are correspondingly denoted by $\mathcal{V}_h$,
$\mathcal{V}_h^i$, $\mathcal{V}_h^b$, $\mathcal{E}_h$, $\mathcal{E}_h^i$ and $\mathcal{E}_h^b$.
For each $E\in\mathcal{E}_h$, $\bm{n}_E$ is a fixed unit vector perpendicular to $E$
and $\bm{t}_E$ is a vector obtained by rotating $\bm{n}_E$ by ninety degree counterclockwisely.
Moreover, for $E\in\mathcal{E}_h^i$, the jump of a function $v$ across $E$ is defined as $[v]_E=v|_{K_1}-v|_{K_2}$,
where $K_1$ and $K_2$ are the cells sharing $E$ as a common edge, and $\bm{n}_E$ points from $K_1$ to $K_2$.
For $E\in\mathcal{E}_h^b$, we set $[v]_E=v|_{K}$ if $E$ is an edge of $K$.
Under these notations, we introduce the following finite element spaces over $\Omega$:
\[
\begin{aligned}
W_h=\Bigg\{w&\in L^2(\Omega):~w|_K\in W_K,~\forall K\in\mathcal{T}_h,
~\mbox{$w$ is continuous at all $V\in\mathcal{V}_h^i$ and}\\
&\mbox{vanishes at all $V\in\mathcal{V}_h^b$},
~\int_E\left[w\right]_E\,\mathrm{d}s
=\int_E\left[\frac{\partial w}{\partial\bm{n}_E}\right]_E\,\mathrm{d}s=0
~\mbox{for all $E\in\mathcal{E}_h$}\Bigg\};\\
\bm{V}_h=\Bigg\{\bm{v}&\in [L^2(\Omega)]^2:~\bm{v}|_K\in\bm{V}_K,~\forall K\in\mathcal{T}_h,\\
&\int_E[\bm{v}\cdot\bm{n}_E]_E\xi\,\mathrm{d}s=0,~\forall\xi\in P_1(E),~\int_E[\bm{v}\cdot\bm{t}_E]_E\,\mathrm{d}s=0,~\forall E\in\mathcal{E}_h\Bigg\};\\
P_h=\Big\{q&\in L_0^2(\Omega):~q|_K\in P_0(K),~\forall K\in\mathcal{T}_h\Big\}.
\end{aligned}
\]

For each $K\in\mathcal{T}_h$,
we define interpolation operators $\mathcal{I}_K:H^2(K)\rightarrow W_K$ and $\bm{\Pi}_K:[H^1(K)]^2\rightarrow \bm{V}_K$,
such that $\tau_{i,K}(\mathcal{I}_Kw)=\tau_{i,K}(w)$, $\sigma_{i,K}(\bm{\Pi}_K\bm{v})=\sigma_{i,K}(\bm{v})$, $i=1,2,\ldots,12$.
The global interpolation operators $\mathcal{I}_h:H_0^2(\Omega)\rightarrow W_h$ and $\bm{\Pi}_h:[H_0^1(\Omega)]^2\rightarrow \bm{V}_h$ are naturally set as $\mathcal{I}_h|_K=\mathcal{I}_K$ and $\bm{\Pi}_h|_K=\bm{\Pi}_K$, $\forall K\in \mathcal{T}_h$.
Similarly, we write $\mathcal{P}_h$ as the $L^2$-projection operator from $L^2(\Omega)$ to $P_h$.
Moreover, the differential operators $\curl$ and $\mathrm{div}$ have their discrete versions $\curl_h$ and $\mathrm{div}_h$,
respectively, determined by $\curl_h|_K=\curl$ on $K$, and $\mathrm{div}_h|_K=\mathrm{div}$ on $K$, $\forall K\in \mathcal{T}_h$. The following commutative diagram provides a discrete Stokes complex.

\begin{theorem}
\label{th: discrete complex}
It holds that
\begin{equation}
\label{e: discrete complex}
\begin{tikzcd}[column sep=large, row sep=large]
0 \arrow{r} & H_0^2(\Omega) \arrow{r}{\curl} \arrow{d}{\mathcal{I}_h}
& \left[H_0^1(\Omega)\right]^2 \arrow{r}{\mathrm{div}} \arrow{d}{\bm{\Pi}_{h}} &
L_0^2(\Omega)\arrow{r} \arrow{d}{\mathcal{P}_h} &0\\
0 \arrow{r} & W_h \arrow{r}{\curl_h}
& \bm{V}_h \arrow{r}{\mathrm{div}_h}& P_h \arrow{r}&0.
\end{tikzcd}
\end{equation}
Moreover, the sequence in the second row is exact.
\end{theorem}
\begin{proof}
Please see Theorem 4.1 in \cite{Zhou2018} for quite a similar argument.
\end{proof}

Define the semi-norms
\[
|v|_{m,h}=\left(\sum_{K\in\mathcal{T}_h}|v|_{m,K}^2\right)^{1/2},~m=0,1,2,3,
\]
then $|\cdot|_{2,h}$ and $|\cdot|_{1,h}$ are norms over $W_h$ and $\bm{V}_h$, respectively.
Since $\mathcal{I}_K$ preserves $P_3(K)$ for each $K\in\mathcal{T}_h$, the Bramble-Hilbert lemma gives
\begin{equation}
\label{e: interp err scalar}
|w-\mathcal{I}_hw|_{j,h}\leq Ch^{k-j}|w|_k,~k=3,4,~j=0,1,2,3,~\forall w\in H_0^2(\Omega)\cap H^4(\Omega).
\end{equation}
Moreover, if $\bm{v}\in [H_0^1(\Omega)\cap H^3(\Omega)]^2$ and $\mathrm{div}\,\bm{v}=0$,
we can find $\bm{v}=\curl\,\phi$ for some $\phi\in H_0^2(\Omega)\cap H^4(\Omega)$.
The commutative diagram (\ref{e: discrete complex}) implies
\begin{equation}
\label{e: interp err vector}
\left|\bm{v}-\bm{\Pi}_h\bm{v}\right|_{1,h}=\left|\curl\,\phi-\curl_h\mathcal{I}_h\phi\right|_{1,h}
\leq C|\phi-\mathcal{I}_h\phi|_{2,h}\leq Ch^2|\phi|_4.
\end{equation}

For a cell $K\in\mathcal{T}_h$,
we write $\mathcal{P}_k^K$  as the $L^2$-projection operator from $L^2(K)$ to $P_k(K)$.
If $E$ is an edge of $K$,
we also define $\mathcal{P}_k^E$ that
$L^2$-projects the traces of the functions in $H^1(K)$ on $E$ to $P_k(E)$.

\section{High accuracy analysis of $W_h$ for biharmonic problems}
\label{s: scalar}

Let us turn to the application of $W_h$ to the biharmonic problem.
Given $f\in H^{-2}(\Omega)$, the biharmonic problem is to find $u$ such that
\begin{equation}
\label{e: model problem scalar}
\begin{aligned}
&\Delta^2u=f~~~&\mbox{in}~\Omega,\\
&u=\frac{\partial u}{\partial\bm{n}}=0~&\mbox{on}~\partial\Omega.
\end{aligned}
\end{equation}
A weak form can be represented as: Find $u\in H_0^2(\Omega)$ such that
\begin{equation}
\label{e: weak scalar}
(\nabla^2 u,\nabla^2 v)=(f,v),~\forall v\in H_0^2(\Omega),
\end{equation}
where $\nabla^2 v$ is the Hessian matrix of $v$.
Note that we have fixed the Poisson ratio $\sigma=0$.
For other $\sigma\in(0,1/2]$, the coming analysis has no intrinsic difference.
The discrete form of (\ref{e: weak scalar}) reads:
Find $u_h\in W_h$ such that
\begin{equation}
\label{e: discrete weak scalar}
\sum_{K\in\mathcal{T}_h}(\nabla^2 u_h,\nabla^2 v_h)_K=(f,v_h),~\forall v_h\in W_h.
\end{equation}
This problem has a unique solution due to the Lax-Milgram theorem.
Note that
\begin{equation}
\label{e: normal weak scalar}
\int_E\left[\frac{\partial w_h}{\partial \bm{n}_E}\right]_E\,\mathrm{d}s=0,~\forall E\in\mathcal{E}_h,~\forall w_h\in W_h.
\end{equation}
Hence, based on a classical consistency error analysis,
one can only predict an $O(h)$ convergence order in discrete $H^2$-norm.
Nevertheless, the result below hints a higher accuracy.

\begin{theorem}
\label{th: converge scalar}
Let $u\in H_0^2(\Omega)\cap H^4(\Omega)$ and $u_h\in W_h$ be the solutions of (\ref{e: weak scalar})
and (\ref{e: discrete weak scalar}), respectively.
If $\mathcal{T}_h$ is uniform, then
\begin{equation}
\label{e: err scalar}
|u-u_h|_{2,h}\leq Ch^2|u|_4.
\end{equation}
\end{theorem}
\begin{proof}
Let us begin with the Strang lemma
\begin{equation}
\label{e: Strang}
|u-u_h|_{2,h}\leq C\left(\inf_{v_h\in W_h}|u-v_h|_{2,h}
+\sup_{w_h\in W_h}\frac{|E_{2,h}(u,w_h)|}{|w_h|_{2,h}}\right)
\end{equation}
with
\[
E_{2,h}(u,w_h)=\sum_{K\in\mathcal{T}_h}(\nabla^2 u,\nabla^2 w_h)_K-(f,w_h).
\]
It follows from (\ref{e: interp err scalar}) that
\[
\inf_{v_h\in W_h}|u-v_h|_{2,h}\leq|u-\mathcal{I}_hu|_{2,h}\leq Ch^2|u|_4,
\]
therefore it suffices to consider the consistency error $E_{2,h}(u,w_h)$.

If $u\in H^4(\Omega)$ then $f\in L^2(\Omega)$.
Let $\mathcal{I}^{S_2}_K$ be the nodal interpolation operator from $W_K$ to $S_2(K)$,
the second order serendipity element space $P_2(K)\oplus\mathrm{span}\{x^2y,xy^2\}$,
satisfying
\[
\mathcal{I}^{S_2}_K(w_h|_K)(V_{i,K})=(w_h|_K)(V_{i,K}),
~\int_{E_{i,K}}\mathcal{I}^{S_2}_K(w_h|_K)\,\mathrm{d}s=\int_{E_{i,K}}w_h|_K\,\mathrm{d}s,~i=1,2,3,4,
\]
and then set $\mathcal{I}^{S_2}_h$ via $\mathcal{I}^{S_2}_h|_K=\mathcal{I}^{S_2}_K$, $\forall K\in\mathcal{T}_h$.
Owing to the weak continuity of $w_h$, we find $\mathcal{I}^{S_2}_hw_h\in H_0^1(\Omega)$.
An integration by parts using (\ref{e: model problem scalar}) gives
\begin{equation}
\label{e: E2h}
\begin{aligned}
E_{2,h}(u,w_h)=&\sum_{K\in\mathcal{T}_h}(\nabla^2 u,\nabla^2 w_h)_K-(f,\mathcal{I}^{S_2}_hw_h)-(f,w_h-\mathcal{I}^{S_2}_hw_h)\\
=&-\sum_{K\in\mathcal{T}_h}(\nabla\Delta u,\nabla(w_h-\mathcal{I}^{S_2}_hw_h))_K-(f,w_h-\mathcal{I}^{S_2}_hw_h)\\
&+\sum_{K\in\mathcal{T}_h}\int_{\partial K}\frac{\partial^2u}{\partial\bm{n}^2}\frac{\partial w_h}{\partial\bm{n}}\,\mathrm{d}s+\sum_{K\in\mathcal{T}_h}\int_{\partial K}\frac{\partial^2u}{\partial\bm{n}\partial\bm{t}}\frac{\partial w_h}{\partial\bm{t}}\,\mathrm{d}s\\
:=&I_1(u,w_h)+I_2(u,w_h)+I_3(u,w_h)+I_4(u,w_h).
\end{aligned}
\end{equation}
As far as $I_1(u,w_h)$ is concerned, noting that
\[
\int_K\nabla(w_h-\mathcal{I}^{S_2}_hw_h)\,\mathrm{d}\bm{x}=\int_{\partial K}(w_h-\mathcal{I}^{S_2}_hw_h)\bm{n}\,\mathrm{d}s=\bm{0}
\]
and $\mathcal{I}^{S_2}_K$ preserves $P_2(K)$, we derive
\begin{equation}
\label{e: I1}
\begin{aligned}
|I_1(u,w_h)|&=\left|\sum_{K\in\mathcal{T}_h}\left(\nabla\Delta u-\mathcal{P}_0^K\nabla\Delta u,
\nabla(w_h-\mathcal{I}^{S_2}_hw_h)\right)_K\right|\\
&\left\{
\begin{aligned}
&\leq\sum_{K\in\mathcal{T}_h}Ch|\nabla\Delta u|_{1,K}Ch|w_h|_{2,K}\leq Ch^2|u|_4|w_h|_{2,h};\\
&\leq\sum_{K\in\mathcal{T}_h}Ch|\nabla\Delta u|_{1,K}Ch^2|w_h|_{3,K}\leq Ch^3|u|_4|w_h|_{3,h}.
\end{aligned}
\right.
\end{aligned}
\end{equation}
One can also estimate $|I_2(u,w_h)|$ as
\begin{equation}
\label{e: I2}
|I_2(u,w_h)|\leq \|f\|_0\|w_h-\Pi_hw_h\|_0
\left\{
\begin{aligned}
&\leq Ch^2\|f\|_0|w_h|_{2,h}\leq Ch^2|u|_4|w_h|_{2,h};\\
&\leq Ch^3\|f\|_0|w_h|_{3,h}\leq Ch^3|u|_4|w_h|_{3,h}.
\end{aligned}
\right.
\end{equation}
Moreover, since the vertex values and the edge integrals of $w_h$ are continuous,
integrating by parts results in
\begin{equation}
\label{e: tangential weak scalar}
\int_E\left[\frac{\partial w_h}{\partial \bm{t}_E}\right]_E\xi_E\,\mathrm{d}s=0,~\forall \xi_E\in P_1(E),
~\forall E\in\mathcal{E}_h,~\forall w_h\in W_h,
\end{equation}
which asserts that
\begin{equation}
\label{e: I4}
\begin{aligned}
|I_4(u,w_h)|&=\left|\sum_{K\in\mathcal{T}_h}\sum_{E\subset\partial K}\int_E
\left(\frac{\partial^2u}{\partial\bm{n}\partial\bm{t}}-\mathcal{P}_1^E\frac{\partial^2u}{\partial\bm{n}\partial\bm{t}}
\right)\left(\frac{\partial w_h}{\partial \bm{t}}-\mathcal{P}_1^E\frac{\partial w_h}{\partial \bm{t}}\right)\,\mathrm{d}\sigma\right|\\
&\leq\sum_{K\in\mathcal{T}_h}\sum_{E\subset\partial K}\left\|\frac{\partial^2u}{\partial\bm{n}\partial\bm{t}}
-\mathcal{P}_1^E\frac{\partial^2u}{\partial\bm{n}\partial\bm{t}}\right\|_{0,E}\left\|\frac{\partial w_h}{\partial \bm{t}}-\mathcal{P}_1^E\frac{\partial w_h}{\partial \bm{t}}\right\|_{0,E}\\
&\left\{
\begin{aligned}
&\leq\sum_{K\in\mathcal{T}_h}Ch^{3/2}|u|_{4,K}Ch^{1/2}|w_h|_{2,K}\leq Ch^2|u|_4|w_h|_{2,h};\\
&\leq\sum_{K\in\mathcal{T}_h}Ch^{3/2}|u|_{4,K}Ch^{3/2}|w_h|_{3,K}\leq Ch^3|u|_4|w_h|_{3,h}.
\end{aligned}
\right.
\end{aligned}
\end{equation}
Hence, it remains to estimate $|I_3(u,w_h)|$.

To this end, notice from the weak continuity of $w_h$ that
\begin{equation}
\label{e: I3}
\begin{aligned}
I_3(u,w_h)=&\sum_{K\in\mathcal{T}_h}\left(
\int_{E_{3,K}}\frac{\partial^2u}{\partial y^2}
\left(\frac{\partial w_h}{\partial y}-\mathcal{P}_0^{E_{3,K}}\frac{\partial w_h}{\partial y}\right)\,\mathrm{d}s
-\int_{E_{1,K}}\frac{\partial^2u}{\partial y^2}
\left(\frac{\partial w_h}{\partial y}-\mathcal{P}_0^{E_{1,K}}\frac{\partial w_h}{\partial y}\right)\,\mathrm{d}s\right)\\
&+\sum_{K\in\mathcal{T}_h}\left(
\int_{E_{2,K}}\frac{\partial^2u}{\partial x^2}
\left(\frac{\partial w_h}{\partial x}-\mathcal{P}_0^{E_{2,K}}\frac{\partial w_h}{\partial x}\right)\,\mathrm{d}s
-\int_{E_{4,K}}\frac{\partial^2u}{\partial x^2}
\left(\frac{\partial w_h}{\partial x}-\mathcal{P}_0^{E_{4,K}}\frac{\partial w_h}{\partial x}\right)\,\mathrm{d}s\right)\\
:=&\sum_{K\in\mathcal{T}_h}I_{1,K}(u,w_h)+\sum_{K\in\mathcal{T}_h}I_{2,K}(u,w_h).
\end{aligned}
\end{equation}
Let us consider the summation involving $I_{1,K}(u,w_h)$. Clearly,
\begin{equation}
\label{e: bilinear scalar est 1}
\begin{aligned}
|I_{1,K}(u,w_h)|&=\Bigg|
\int_{E_{3,K}}\left(\frac{\partial^2u}{\partial y^2}-\mathcal{P}_0^{E_{3,K}}\frac{\partial^2u}{\partial y^2}\right)
\left(\frac{\partial w_h}{\partial y}-\mathcal{P}_0^{E_{3,K}}\frac{\partial w_h}{\partial y}\right)\,\mathrm{d}s\\
&~~~~~-\int_{E_{1,K}}\left(\frac{\partial^2u}{\partial y^2}-\mathcal{P}_0^{E_{1,K}}\frac{\partial^2u}{\partial y^2}\right)
\left(\frac{\partial w_h}{\partial y}-\mathcal{P}_0^{E_{1,K}}\frac{\partial w_h}{\partial y}\right)\,\mathrm{d}s
\Bigg|\\
&\leq Ch\left|\frac{\partial^2u}{\partial y^2}\right|_{1,K}|w_h|_{2,K}.
\end{aligned}
\end{equation}
On the other hand, introduce the bilinear form $I'_{1,K}(u,w_h)$:
\begin{equation}
\label{e: I'}
I'_{1,K}(u,w_h)=\frac{h_x^2}{12}\int_K\frac{\partial^3u}{\partial x\partial y^2}\frac{\partial^3w_h}{\partial x\partial y^2}
+\frac{\partial^4u}{\partial x^2\partial y^2}\frac{\partial^2w_h}{\partial y^2}\,\mathrm{d}x\mathrm{d}y.
\end{equation}
It follows from the inverse inequality that
\begin{equation}
\label{e: bilinear scalar est 2}
\begin{aligned}
|I'_{1,K}(u,w_h)|&\leq Ch^2\left|\frac{\partial^2u}{\partial y^2}\right|_{1,K}|w_h|_{3,K}+Ch^2\left|\frac{\partial^2u}{\partial y^2}\right|_{2,K}|w_h|_{2,K}\\
&\leq Ch\left(\left|\frac{\partial^2u}{\partial y^2}\right|_{1,K}+h\left|\frac{\partial^2u}{\partial y^2}\right|_{2,K}\right)|w_h|_{2,K}.
\end{aligned}
\end{equation}
Moreover, we find
\begin{equation}
\label{e: equal scalar y}
I_{1,K}(u,w_h)=I'_{1,K}(u,w_h)=0,~\mbox{if}~\frac{\partial^2 u}{\partial y^2}=y.
\end{equation}
Next, an easy verification shows
\[
\frac{\partial w_h}{\partial y}\bigg|_{E_{i,K}}\in P_2(E_{i,K}),~i=1,3,
\]
therefore by the Simpson quadrature rule,
\begin{equation}
\label{e: normal relation scalar}
\begin{aligned}
\frac{1}{h_x}\int_{E_{i,K}}&\left(\frac{\partial w_h}{\partial y}-\mathcal{P}_0^{E_{i,K}}\frac{\partial w_h}{\partial y}\right)\xi_{i,K}\,\mathrm{d}s=\frac{1}{h_x}\int_{E_{i,K}}\frac{\partial w_h}{\partial y}\xi_{i,K}\,\mathrm{d}s\\
&=\frac{1}{6}\left(\frac{\partial w_h}{\partial y}(V_{i'',K})-\frac{\partial w_h}{\partial y}(V_{i',K})\right)
=\frac{1}{6}\int_{E_{i,K}}\frac{\partial^2 w_h}{\partial x\partial y}\,\mathrm{d}s,
\end{aligned}
\end{equation}
where $i''=3$ and $i'=4$ for $i=3$, and $i''=2$ and $i'=1$ for $i=1$.
This gives
\begin{equation}
\label{e: derivation x scalar}
\begin{aligned}
\frac{1}{h_x}&\left(\int_{E_{3,K}}\left(\frac{\partial w_h}{\partial y}-\mathcal{P}_0^{E_{3,K}}\frac{\partial w_h}{\partial y}\right)\xi_{3,K}\,\mathrm{d}s
-\int_{E_{1,K}}\left(\frac{\partial w_h}{\partial y}-\mathcal{P}_0^{E_{1,K}}\frac{\partial w_h}{\partial y}\right)\xi_{1,K}\,\mathrm{d}s
\right)\\
&=\frac{1}{6}\int_{x'_K}^{x''_K}\frac{\partial^2 w_h}{\partial x\partial y}\bigg|_{y=y''_K}
-\frac{\partial^2 w_h}{\partial x\partial y}\bigg|_{y=y'_K}\,\mathrm{d}x
=\frac{1}{6}\int_K\frac{\partial^3 w_h}{\partial x\partial y^2}\,\mathrm{d}x\mathrm{d}y\\
&=\frac{h_x}{12}\int_K\frac{\partial\xi_{3,K}}{\partial x}\frac{\partial^3 w_h}{\partial x\partial y^2}\,\mathrm{d}x\mathrm{d}y,
\end{aligned}
\end{equation}
and thus
\begin{equation}
\label{e: equal scalar x}
I_{1,K}(u,w_h)=I'_{1,K}(u,w_h),~\mbox{if}~\frac{\partial^2 u}{\partial y^2}=x.
\end{equation}
It follows from (\ref{e: equal scalar y}) and (\ref{e: equal scalar x}) that
\begin{equation}
\label{e: key 1}
I_{1,K}(u,w_h)-I'_{1,K}(u,w_h)=0,~\forall \frac{\partial^2u}{\partial y^2}\in P_1(K),~\forall w_h\in W_K.
\end{equation}
Owing to (\ref{e: bilinear scalar est 1}), (\ref{e: bilinear scalar est 2}), (\ref{e: key 1}) and the Bramble-Hilbert lemma, we derive
\begin{equation}
\label{e: bilinear scalar est 3}
|I_{1,K}(u,w_h)-I'_{1,K}(u,w_h)|\leq Ch^2\left|\frac{\partial^2u}{\partial y^2}\right|_{2,K}|w_h|_{2,K}
\leq Ch^2|u|_{4,K}|w_h|_{2,K}.
\end{equation}
Substituting (\ref{e: bilinear scalar est 3}) into (\ref{e: I3}) results in
\begin{equation}
\label{e: I1K}
\sum_{K\in\mathcal{T}_h}I_{1,K}(u,w_h)\leq \sum_{K\in\mathcal{T}_h}I'_{1,K}(u,w_h)+Ch^2|u|_4|w_h|_{2,h}.
\end{equation}

Now we are in the position to estimate the right-hand side in (\ref{e: I1K}).
Indeed,
\[
\begin{aligned}
\sum_{K\in\mathcal{T}_h}I'_{1,K}(u,w_h)&=\frac{h_x^2}{12}\sum_{K\in\mathcal{T}_h}\int_{y'_K}^{y''_K}\left(\int_{x'_K}^{x''_K}
\frac{\partial^3u}{\partial x\partial y^2}\frac{\partial^3w_h}{\partial x\partial y^2}+\frac{\partial^4u}{\partial x^2\partial y^2}\frac{\partial^2w_h}{\partial y^2}\,\mathrm{d}x\right)\mathrm{d}y\\
&=\frac{h_x^2}{12}\sum_{K\in\mathcal{T}_h}\int_{y'_K}^{y''_K}\frac{\partial^3u}{\partial x\partial y^2}\frac{\partial^2w_h}{\partial y^2}\bigg|_{x=x''_K}-\frac{\partial^3u}{\partial x\partial y^2}\frac{\partial^2w_h}{\partial y^2}\bigg|_{x=x'_K}\,\mathrm{d}y\\
&=\frac{h_x^2}{12}\sum_{K\in\mathcal{T}_h}\Bigg(\int_{E_{2,K}}\frac{\partial^3u}{\partial x\partial y^2}
\left(\frac{\partial^2w_h}{\partial y^2}-\frac{\partial^2\mathcal{I}^{S_2}_Kw_h}{\partial y^2}\right)\,\mathrm{d}s\\
&~~~~~~~~~~~~~~~~~~-\int_{E_{4,K}}\frac{\partial^3u}{\partial x\partial y^2}
\left(\frac{\partial^2w_h}{\partial y^2}-\frac{\partial^2\mathcal{I}^{S_2}_Kw_h}{\partial y^2}\right)\,\mathrm{d}s\Bigg),
\end{aligned}
\]
where we have used the integration by parts formula,
the weak continuity of $w_h$ and the uniformity of the partition $\mathcal{T}_h$.
For each $K\in\mathcal{T}_h$,
\[
\begin{aligned}
(w_h|_K)|_{E_{2,K}}&\in P_2(E_{2,K})\oplus\mathrm{span}\{x^3,x^4\},\\
(w_h|_K)|_{E_{4,K}}&\in P_2(E_{4,K})\oplus\mathrm{span}\{x^3,x^4\}.
\end{aligned}
\]
Note that $x^3$ and $x^4$ are independent of $y$,
therefore a closer observation for $\mathcal{I}^{S_2}_h$ gives
\begin{equation}
\label{e: key 2}
(w_h|_K-\mathcal{I}^{S_2}_hw_h)\big|_{E_{2,K}}=(w_h|_K-\mathcal{I}^{S_2}_hw_h)\big|_{E_{4,K}},~\forall w_h\in W_h.
\end{equation}
Hence, it follows from a standard argument and the inverse inequality that
\begin{equation}
\label{e: tangential estimate scalar}
\begin{aligned}
\left|\sum_{K\in\mathcal{T}_h}I'_{1,K}(u,w_h)\right|&=\frac{h_x^2}{12}\Bigg|\sum_{K\in\mathcal{T}_h}\Bigg(\int_{E_{2,K}}
\left(\frac{\partial^3u}{\partial x\partial y^2}-\mathcal{P}_0^K\frac{\partial^3u}{\partial x\partial y^2}\right)
\left(\frac{\partial^2w_h}{\partial y^2}-\frac{\partial^2\mathcal{I}^{S_2}_hw_h}{\partial y^2}\right)\,\mathrm{d}s\\
&~~~~~~~~~~~~~~~~~~-\int_{E_{4,K}}
\left(\frac{\partial^3u}{\partial x\partial y^2}-\mathcal{P}_0^K\frac{\partial^3u}{\partial x\partial y^2}\right)
\left(\frac{\partial^2w_h}{\partial y^2}-\frac{\partial^2\mathcal{I}^{S_2}_hw_h}{\partial y^2}\right)\,\mathrm{d}s\Bigg)\Bigg|\\
&\leq Ch^2\sum_{K\in\mathcal{T}_h}\sum_{i=2,4}Ch^{1/2}\left|\frac{\partial^3u}{\partial x\partial y^2}\right|_{1,K}
Ch^{-2}\|w_h-\mathcal{I}^{S_2}_Kw_h\|_{0,E_{i,K}}\\
&\left\{
\begin{aligned}
&\leq Ch^2\sum_{K\in\mathcal{T}_h}\sum_{i=2,4}Ch^{1/2}|u|_{4,K}Ch^{-2}Ch^{3/2}|w_h|_{2,K}\leq Ch^2|u|_4|w_h|_{2,h};\\
&\leq Ch^2\sum_{K\in\mathcal{T}_h}\sum_{i=2,4}Ch^{1/2}|u|_{4,K}Ch^{-2}Ch^{5/2}|w_h|_{3,K}\leq Ch^3|u|_4|w_h|_{3,h}.
\end{aligned}
\right.
\end{aligned}
\end{equation}
Combining this estimate with (\ref{e: I1K}) we get
\[
\left|\sum_{K\in\mathcal{T}_h}I_{1,K}(u,w_h)\right|\leq Ch^2|u|_4|w_h|_{2,h}.
\]
Similarly, it holds that
\[
\left|\sum_{K\in\mathcal{T}_h}I_{2,K}(u,w_h)\right|\leq Ch^2|u|_4|w_h|_{2,h},
\]
which along with (\ref{e: I1}), (\ref{e: I2}), (\ref{e: I4}) and (\ref{e: I3}) gives
\[
|E_{2,h}(u,w_h)|\leq Ch^2|u|_4|w_h|_{2,h},
\]
and the proof is completed.
\end{proof}

Thanks to the zeroth order weak continuity for the gradients of the functions in $W_h$,
we can obtain the discrete $H^1$ error estimate using the duality argument provided in \cite{Shi1986,Park2013}.
To proceed, we need to construct an auxiliary $C^0$ nonconforming element with the first order orthogonality for normal derivatives, which will be carried out in two steps.
In the first step,
define the local shape function space
\[
\widetilde{W}_K^*=P_3(K)\oplus\mathrm{span}\{x^3y,xy^3\}\oplus\mathrm{span}\{b_K\phi_{i,K},~i=1,2,\ldots,8\},~K\in\mathcal{T}_h,
\]
where $b_K$ is a 4-degree bubble polynomial fulfilling $b_K|_{\partial K}=0$.
The functions $\phi_{i,K}$ are selected such that $\widetilde{W}_K^*$ is unisolvent to the DoF set
\[
\widetilde{T}_K^*=\left\{v(V_{i,K}),\nabla v(V_{i,K}),\int_{E_{i,K}}\frac{\partial v}{\partial \bm{n}}\,\mathrm{d}s,
\int_{E_{i,K}}\frac{\partial v}{\partial \bm{n}}\xi_{i,K}\,\mathrm{d}s,~i=1,2,3,4\right\},~\forall v\in \widetilde{W}_K^*.
\]
According to \cite{Chen2012}, a successful selection of these $\phi_{i,\widehat{K}}$ over the reference square $\widehat{K}=[-1,1]^2$ can be given as
\[
\widehat{x},\widehat{y},\widehat{x}^2\widehat{y},\widehat{x}\widehat{y}^2,
\widehat{x}^4,\widehat{y}^4,\widehat{x}^3\widehat{y},\widehat{x}\widehat{y}^3,
\]
then the physical $\phi_{i,K}$ will be obtained by an affine equivalent technique, $i=1,2,\ldots,8$.
The second step is to take a subspace of $\widetilde{W}_K^*$, named as $W_K^*$,
satisfying the following four linear relations of the DoFs in $\widetilde{T}_K^*$:
\begin{equation}
\label{e: linear relations scalar}
\begin{aligned}
\frac{1}{h_x}\int_{E_{i,K}}\frac{\partial v}{\partial y}\xi_{i,K}\,\mathrm{d}s
&=\frac{1}{6}\left(\frac{\partial v}{\partial y}(V_{i'',K})-\frac{\partial v}{\partial y}(V_{i',K})\right),~i=1,3,\\
\frac{1}{h_y}\int_{E_{j,K}}\frac{\partial v}{\partial x}\xi_{j,K}\,\mathrm{d}s
&=\frac{1}{6}\left(\frac{\partial v}{\partial x}(V_{j'',K})-\frac{\partial v}{\partial x}(V_{j',K})\right),~j=2,4,
~\forall v\in W_K^*,
\end{aligned}
\end{equation}
where $i''=3$ and $i'=4$ for $i=3$, and $i''=2$ and $i'=1$ for $i=1$;
$j''=3$ and $j'=2$ for $j=2$, and $j''=4$ and $j'=1$ for $j=4$.
By the Simpson quadrature rule, we find $P_3(K)\subset W_K^*$.
The actual DoF set is then set as
\[
T_K^*=\left\{v(V_{i,K}),\nabla v(V_{i,K}),\int_{E_{i,K}}\frac{\partial v}{\partial \bm{n}}\,\mathrm{d}s,~i=1,2,3,4\right\},~\forall v\in W_K^*,
\]
and the global finite element space $W_h^*$ reads
\[
\begin{aligned}
W_h^*=\Bigg\{w&\in L^2(\Omega):~w|_K\in W_K^*,~\forall K\in\mathcal{T}_h,
~\mbox{$w$ and $\nabla w$ are continuous at all}\\
&\mbox{$V\in\mathcal{V}_h^i$ and vanish at all $V\in\mathcal{V}_h^b$},
~\int_E\left[\frac{\partial w}{\partial\bm{n}_E}\right]_E\,\mathrm{d}s=0
~\mbox{for all $E\in\mathcal{E}_h$}\Bigg\}.
\end{aligned}
\]
Clearly, $W_h^*\subset H_0^1(\Omega)$.
For a fixed $K$, the interpolation $\mathcal{I}_K^*$ from $H^3(K)$ to $W_K^*$ is defined such that $\tau(\mathcal{I}_K^*v)=\tau(v)$,
$\forall v\in H^3(K)$, $\forall \tau\in T_K^*$.
The global version $\mathcal{I}_h^*$ from $H^3(\Omega)\cap H_0^2(\Omega)$ to $W_h^*$ is naturally determined by $\mathcal{I}_h^*|_K=\mathcal{I}_K^*$, $\forall K\in\mathcal{T}_h$.
This element is indeed a rectangular extension of the triangular $C^0$ nonconforming element designed by Gao et al.~\cite{Gao2011}.

\begin{theorem}
\label{th: converge scalar H1}
Let the solution domain $\Omega$ be convex.
In addition, under the same assumptions in Theorem \ref{th: converge scalar}, we have
\[
|u-u_h|_{1,h}\leq Ch^3|u|_4.
\]
\end{theorem}
\begin{proof}
We first set $e_h=u-u_h$ and consider the following dual problem: Find $\phi\in H_0^2(\Omega)\cap H^3(\Omega)$,
such that
\begin{equation}
\label{e: dual problem scalar}
(\nabla^2 \phi,\nabla^2 v)=\sum_{K\in\mathcal{T}_h}(\nabla e_h,\nabla v)_K,~\forall v\in H_0^2(\Omega).
\end{equation}
For the convex domain $\Omega$,
it follows from \cite{Grisvard1985} that $\Delta^2$ is an isomorphism from $H_0^2(\Omega)\cap H^3(\Omega)$ onto $H^{-1}(\Omega)$, and therefore (\ref{e: dual problem scalar}) has a unique solution with the bound
\begin{equation}
\label{e: dual regular scalar}
|\phi|_3\leq C\sup_{v\in H_0^1(\Omega)}\frac{\sum_{K\in\mathcal{T}_h}(\nabla e_h,\nabla v)_K}{|v|_1}\leq C|e_h|_{1,h}.
\end{equation}
We will use the decomposition
\begin{equation}
\label{e: decompose scalar}
|e_h|_{1,h}^2=\sum_{K\in\mathcal{T}_h}\left(\nabla e_h,\nabla (e_h-\mathcal{I}^{S_2}_he_h)\right)_K+\sum_{K\in\mathcal{T}_h}\left(\nabla e_h,\nabla \mathcal{I}^{S_2}_he_h\right)_K.
\end{equation}
The first summation is estimated with the aid of (\ref{e: err scalar}) as
\begin{equation}
\label{e: dual part 1 scalar}
\begin{aligned}
\left|\sum_{K\in\mathcal{T}_h}\left(\nabla e_h,\nabla (e_h-\mathcal{I}^{S_2}_he_h)\right)_K\right|
&\leq C|e_h|_{1,h}|e_h-\mathcal{I}^{S_2}_he_h|_{1,h}\\
&\leq Ch|e_h|_{1,h}|e_h|_{2,h}\leq Ch^3|u|_4|e_h|_{1,h}.
\end{aligned}
\end{equation}
As far as the second summation in (\ref{e: decompose scalar}) is concerned,
note that $\mathcal{I}^{S_2}_he_h\in H_0^1(\Omega)$,
and a density argument using (\ref{e: dual problem scalar}) with an integration
by parts shows that
\begin{equation}
\label{e: dual error part 2 main scalar}
\begin{aligned}
\left|\sum_{K\in\mathcal{T}_h}\left(\nabla e_h,\nabla \mathcal{I}^{S_2}_he_h\right)_K\right|
&=\left|-\left(\nabla\Delta\phi,\nabla \mathcal{I}^{S_2}_he_h\right)\right|\\
&=\left|\sum_{K\in\mathcal{T}_h}\left(\nabla\Delta\phi,\nabla (e_h-\mathcal{I}^{S_2}_he_h)\right)_K
-\sum_{K\in\mathcal{T}_h}\left(\nabla\Delta\phi,\nabla e_h\right)_K\right|\\
&\leq C|\phi|_3|e_h-\mathcal{I}^{S_2}_he_h|_{1,h}+|I_3(\phi,e_h)|+|I_4(\phi,e_h)|
+\left|\sum_{K\in\mathcal{T}_h}(\nabla^2 \phi,\nabla^2 e_h)_K\right|\\
&\leq Ch|\phi|_3|e_h|_{2,h}+Ch|\phi|_3|e_h|_{2,h}+\left|\sum_{K\in\mathcal{T}_h}(\nabla^2 \phi,\nabla^2 e_h)_K\right|,
\end{aligned}
\end{equation}
where we have used the weak continuity of $e_h$, see e.g.~(\ref{e: normal weak scalar}) and (\ref{e: tangential weak scalar}). The first two terms in the last line of (\ref{e: dual error part 2 main scalar}) are bounded by
\[
Ch|\phi|_3|e_h|_{2,h}\leq Ch^3|\phi|_3|u|_4\leq Ch^3|u|_4|e_h|_{1,h}
\]
according to (\ref{e: err scalar}) and (\ref{e: dual regular scalar}).
Hence it remains to estimate the last term in the last line of (\ref{e: dual error part 2 main scalar}).

To this end,
we select $\phi^*=\mathcal{I}_h^*\phi\in W_h^*$, and then $\mathcal{I}_h\phi^*\in W_h$ due to the definitions of $W_h^*$ and $W_h$.
As a result,
it follows from (\ref{e: discrete weak scalar}) and (\ref{e: E2h}) that
\begin{equation}
\label{e: dual part 2 2 scalar}
\begin{aligned}
\sum_{K\in\mathcal{T}_h}(\nabla^2 \phi,\nabla^2 e_h)_K
&=\sum_{K\in\mathcal{T}_h}(\nabla^2 \mathcal{I}_h\phi^*,\nabla^2 e_h)_K
+\sum_{K\in\mathcal{T}_h}(\nabla^2(\phi-\mathcal{I}_h\phi^*),\nabla^2 e_h)_K\\
&=\sum_{K\in\mathcal{T}_h}(\nabla^2 u, \nabla^2 \mathcal{I}_h\phi^*)_K
-(f,\mathcal{I}_h\phi^*)+\sum_{K\in\mathcal{T}_h}(\nabla^2(\phi-\mathcal{I}_h\phi^*),\nabla^2 e_h)_K\\
&=E_{2,h}(u,\mathcal{I}_h\phi^*)+\sum_{K\in\mathcal{T}_h}(\nabla^2(\phi-\mathcal{I}_h\phi^*),\nabla^2 e_h)_K\\
&=\sum_{i=1}^4I_i(u,\mathcal{I}_h\phi^*)+\sum_{K\in\mathcal{T}_h}(\nabla^2(\phi-\mathcal{I}_h\phi^*),\nabla^2 e_h)_K.
\end{aligned}
\end{equation}
Invoking (\ref{e: I1}), (\ref{e: I2}) and (\ref{e: I4}), one sees
\begin{equation}
\label{e: collect dual 1 scalar}
|I_i(u,\mathcal{I}_h\phi^*)|\leq Ch^3|u|_4|\mathcal{I}_h\phi^*|_{3,h}\leq Ch^3|u|_4|\phi^*|_{3,h}
\leq Ch^3|u|_4|\phi|_3,~i=1,2,4.
\end{equation}
Furthermore, by the triangle inequality,
\begin{equation}
\label{e: collect dual 2 scalar}
\begin{aligned}
\left|\sum_{K\in\mathcal{T}_h}(\nabla^2(\phi-\mathcal{I}_h\phi^*),\nabla^2 e_h)_K\right|
&\leq C|\phi-\mathcal{I}_h\phi^*|_{2,h}|e_h|_{2,h}\\
&\leq C|\phi-\mathcal{I}_h^*\phi|_{2,h}|e_h|_{2,h}+C|\phi^*-\mathcal{I}_h\phi^*|_{2,h}|e_h|_{2,h}\\
&\leq Ch^3|u|_4|\phi|_3.
\end{aligned}
\end{equation}
Hence, we shall estimate $I_3(u,\mathcal{I}_h\phi^*)$.
Indeed, we write $I_3(u,\mathcal{I}_h\phi^*)$ as
\[
I_3(u,\mathcal{I}_h\phi^*)=I_3(u,\mathcal{I}_h\phi^*-\phi^*)+I_3(u,\phi^*).
\]
Note that the linear relations in (\ref{e: linear relations scalar}) and the weak continuity of $\phi^*$ ensure that
\begin{equation}
\label{e: dual normal weak scalar}
\int_E\left[\frac{\partial \phi^*}{\partial \bm{n}_E}\right]_E\xi_E\,\mathrm{d}s=0,~\forall \xi_E\in P_1(E),
~\forall E\in\mathcal{E}_h,
\end{equation}
which gives
\begin{equation}
\label{e: collect dual 3 scalar}
\begin{aligned}
|I_3(u,\phi^*)|&=\left|\sum_{K\in\mathcal{T}_h}\sum_{E\subset\partial K}\int_E
\left(\frac{\partial^2u}{\partial\bm{n}^2}-\mathcal{P}_1^E\frac{\partial^2u}{\partial\bm{n}^2}
\right)\left(\frac{\partial \phi^*}{\partial \bm{n}}-\mathcal{P}_1^E\frac{\partial \phi^*}{\partial \bm{n}}\right)\,\mathrm{d}\sigma\right|\\
&\leq\sum_{K\in\mathcal{T}_h}Ch^{3/2}|u|_{4,K}Ch^{3/2}|\phi^*|_{3,K}\\
&\leq Ch^3|u|_4|\phi^*|_{3,h}\leq Ch^3|u|_4|\phi|_3.
\end{aligned}
\end{equation}

Next we consider $I_3(u,\mathcal{I}_h\phi^*-\phi^*)$.
Just as in (\ref{e: I3}), by the weak continuity of both $W_h$ and $W_h^*$, we can write
\begin{equation}
\label{e: I3 dual}
I_3(u,\mathcal{I}_h\phi^*-\phi^*)=\sum_{K\in\mathcal{T}_h}I_{1,K}(u,\mathcal{I}_h\phi^*-\phi^*)
+\sum_{K\in\mathcal{T}_h}I_{2,K}(u,\mathcal{I}_h\phi^*-\phi^*).
\end{equation}
Let us investigate the relation between $I_{1,K}(u,\phi^*)$ and $I'_{1,K}(u,\phi^*)$
defined in (\ref{e: I3}) and (\ref{e: I'}) for each $K\in\mathcal{T}_h$.
It is clear that (\ref{e: equal scalar y}) still holds if $w_h$ is replaced by $\phi^*$.
Moreover, a careful comparison between (\ref{e: normal relation scalar}) and the first line in (\ref{e: linear relations scalar}) shows that they are precisely the same thing if $w_h$ in (\ref{e: normal relation scalar}) is replaced by $v$
in (\ref{e: linear relations scalar}).
Hence, one can follow the same argument as in (\ref{e: normal relation scalar}) and (\ref{e: derivation x scalar})
to derive an analogue of (\ref{e: equal scalar x}).
As a consequence,
\[
I_{1,K}(u,\phi^*)-I'_{1,K}(u,\phi^*)=0,~\forall \frac{\partial^2u}{\partial y^2}\in P_1(K),
\]
which asserts by (\ref{e: bilinear scalar est 1}), (\ref{e: bilinear scalar est 2}), (\ref{e: key 1}) and the Bramble-Hilbert lemma that
\begin{equation}
\label{e: normal to tangent dual}
\begin{aligned}
\left|I_{1,K}(u,\mathcal{I}_h\phi^*-\phi^*)-I'_{1,K}(u,\mathcal{I}_h\phi^*-\phi^*)\right|
&\leq Ch^2\left|\frac{\partial^2u}{\partial y^2}\right|_{2,K}|\mathcal{I}_h\phi^*-\phi^*|_{2,K}\\
&\leq Ch^3|u|_{4,K}|\phi^*|_{3,K}\leq Ch^3|u|_{4,K}|\phi|_{3,K}.
\end{aligned}
\end{equation}
Substituting (\ref{e: normal to tangent dual}) into (\ref{e: I3 dual}) results in
\begin{equation}
\label{e: I1K dual}
\sum_{K\in\mathcal{T}_h}I_{1,K}(u,\mathcal{I}_h\phi^*-\phi^*)\leq \sum_{K\in\mathcal{T}_h}I'_{1,K}(u,\mathcal{I}_h\phi^*-\phi^*)+Ch^3|u|_4|\phi|_3.
\end{equation}
The right-hand side of (\ref{e: I1K dual}) can be estimated by an integration by parts:
\[
\begin{aligned}
\sum_{K\in\mathcal{T}_h}I'_{1,K}(u,\mathcal{I}_h\phi^*-\phi^*)
=&\frac{h_x^2}{12}\Bigg(\sum_{K\in\mathcal{T}_h}\int_{y'_K}^{y''_K}\frac{\partial^3u}{\partial x\partial y^2}\frac{\partial^2\mathcal{I}_h\phi^*}{\partial y^2}\bigg|_{x=x''_K}-\frac{\partial^3u}{\partial x\partial y^2}\frac{\partial^2\mathcal{I}_h\phi^*}{\partial y^2}\bigg|_{x=x'_K}\,\mathrm{d}y\\
&~~~~~~~~~-\sum_{K\in\mathcal{T}_h}\int_{y'_K}^{y''_K}\frac{\partial^3u}{\partial x\partial y^2}\frac{\partial^2\phi^*}{\partial y^2}\bigg|_{x=x''_K}-\frac{\partial^3u}{\partial x\partial y^2}\frac{\partial^2\phi^*}{\partial y^2}\bigg|_{x=x'_K}\,\mathrm{d}y
\Bigg).\\
\end{aligned}
\]
The last summation immediately vanishes as $\phi^*\in H_0^1(\Omega)$.
Therefore using the argument in (\ref{e: tangential estimate scalar}), one derives
\[
\left|\sum_{K\in\mathcal{T}_h}I'_{1,K}(u,\mathcal{I}_h\phi^*-\phi^*)\right|
\leq Ch^3|u|_4|\mathcal{I}_h\phi^*|_{3,h}\leq Ch^3|u|_4|\phi|_3,
\]
which along with (\ref{e: I1K dual}) gives
\begin{equation}
\label{e: I1K final dual}
\left|\sum_{K\in\mathcal{T}_h}I_{1,K}(u,\mathcal{I}_h\phi^*-\phi^*)\right|\leq Ch^3|u|_4|\phi|_3.
\end{equation}
In a similar fashion, it holds that
\begin{equation}
\label{e: I2K final dual}
\left|\sum_{K\in\mathcal{T}_h}I_{2,K}(u,\mathcal{I}_h\phi^*-\phi^*)\right|\leq Ch^3|u|_4|\phi|_3.
\end{equation}
Substituting (\ref{e: I1K final dual}), (\ref{e: I2K final dual}) into (\ref{e: I3 dual})
and combining (\ref{e: collect dual 1 scalar}), (\ref{e: collect dual 2 scalar}) and (\ref{e: collect dual 3 scalar})
with (\ref{e: dual part 2 2 scalar}), we have
\[
\left|\sum_{K\in\mathcal{T}_h}(\nabla^2 \phi,\nabla^2 e_h)_K\right|\leq Ch^3|u|_4|\phi|_3\leq Ch^3|u|_4|e_h|_{1,h}.
\]
Note that we have again used the regularity condition (\ref{e: dual regular scalar}).
Finally, it follows from (\ref{e: decompose scalar}), (\ref{e: dual part 1 scalar}) and (\ref{e: dual error part 2 main scalar}) that
\[
|e_h|_{1,h}^2\leq Ch^3|u|_4|e_h|_{1,h}.
\]
The desired conclusion is obtained by dividing both sides by $|e_h|_{1,h}$.
\end{proof}

\begin{remark}
In our previous work \cite{Zhou2018} over general quadrilateral meshes,
the shape function space $W_K$ is of degree six,
so that the edge integrals can be replaced by edge midpoint values in the DoF set,
which seems cheaper in practical applications, especially for nonhomogeneous boundary conditions.
Nevertheless, when applied to uniform rectangular meshes,
the key features (\ref{e: key 1}) and (\ref{e: key 2}) remain valid,
and so the error estimates in Theorems \ref{th: converge scalar} and \ref{th: converge scalar H1} still hold.
\end{remark}

\section{High accuracy analysis of $\bm{V}_h\times P_h$ for Stokes problems}
\label{s: vector}

In a similar fashion, we will give the high accuracy analysis of the Stokes element $\bm{V}_h\times P_h$.
The Stokes problem for incompressible flow reads:
For a given force field $\bm{f}$,
find the velocity $\bm{u}$ and the pressure $p$ satisfying
\begin{equation}
\label{e: stokes problem}
\begin{aligned}
-\Delta\bm{u}+\nabla p&=\bm{f}~~~~~~&\mbox{in}~\Omega,\\
\mbox{div}\,\bm{u}&=0~&\mbox{in}~\Omega,\\
\bm{u}&=\bm{0}~&\mbox{on}~\partial\Omega,
\end{aligned}
\end{equation}
A weak formulation of (\ref{e: stokes problem}) is to find a
pair $(\bm{u},p)\in [H_0^1(\Omega)]^2\times L_0^2(\Omega)$ such that
\begin{equation}
\label{e: stokes continuous}
\begin{aligned}
(\nabla\bm{u},\nabla\bm{v})-(\mathrm{div}\,\bm{v},p)&=(\bm{f},\bm{v}),~&\forall \bm{v}\in[H_0^1(\Omega)]^2,\\
(\mathrm{div}\,\bm{u},q)&=0,~&\forall q\in L_0^2(\Omega).
\end{aligned}
\end{equation}
We will approximate (\ref{e: stokes continuous}) by seeking $(\bm{u}_h,p_h)\in\bm{V}_h\times P_h$ yielding
\begin{equation}
\label{e: stokes discrete}
\begin{aligned}
a_h(\bm{u}_h,\bm{v}_h)-b_h(\bm{v}_h,p_h)&=(\bm{f},\bm{v}_h),~\forall \bm{v}_h\in\bm{V}_h,\\
b_h(\bm{u}_h,q_h)&=0,~~\forall q_h\in P_h
\end{aligned}
\end{equation}
with bilinear forms
\[
a_h(\bm{u}_h,\bm{v}_h)=\sum_{K\in\mathcal{T}_h}(\nabla\bm{u}_h,\nabla\bm{v}_h)_K,
~b_h(\bm{v}_h,q_h)=(\mathrm{div}_h\bm{v}_h,q_h).
\]
Following a classical consistency error analysis,
only an $O(h)$ convergence order can be derived for $|\bm{u}-\bm{u}_h|_{1,h}$,
since one only has the following weak continuity for the tangential component:
\begin{equation}
\label{e: tangential weak stokes}
\int_E\left[\bm{w}_h\cdot\bm{t}_E\right]_E\,\mathrm{d}s=0,~\forall E\in\mathcal{E}_h,~\forall \bm{w}_h\in \bm{V}_h.
\end{equation}
As before, we have the following higher order error estimate.

\begin{theorem}
\label{th: converge stokes}
The discrete problem (\ref{e: stokes discrete}) has a unique solution $(\bm{u}_h,p_h)\in\bm{V}_h\times P_h$.
Let $(\bm{u},p)\in\left([H_0^1(\Omega)\cap H^3(\Omega)]^2\right)\times(L_0^2(\Omega)\cap H^2(\Omega))$ be the weak solution of (\ref{e: stokes continuous}) fulfilling $\bm{u}=\curl\,\phi$, $\phi\in H_0^2(\Omega)\cap H^4(\Omega)$.
If $\mathcal{T}_h$ is uniform, the following estimates hold:
\begin{equation}
\label{e: converge stokes}
|\bm{u}-\bm{u}|_{1,h}\leq Ch^2(|\phi|_4+|p|_2),~\|p-p_h\|_0\leq Ch|p|_1+Ch^2(|\phi|_4+|p|_2).
\end{equation}
\end{theorem}
\begin{proof}
From the commutative diagram (\ref{e: discrete complex}) we see $\mathrm{div}_h\bm{V}_h\subset P_h$ and
\[
\mathrm{div}_h\bm{\Pi}_h\bm{v}=\mathcal{P}_h\mathrm{div}\,\bm{v},~\forall\bm{v}\in [H_0^1(\Omega)]^2.
\]
Hence, using Fortin's trick and the continuous inf-sup condition, we obtain the following discrete version
\[
\sup_{\bm{v}_h\in\bm{V}_h}\frac{b_h(\bm{v}_h,q_h)}{|\bm{v}_h|_{1,h}}\geq C\|q_h\|_0,~\forall q_h\in P_h.
\]
By a standard argument of the mixed finite element method, e.g., Theorem 3.1 in \cite{Zhou2018} or
Theorem 5.2.6 in \cite{Boffi2013}, (\ref{e: stokes discrete}) has a unique solution $(\bm{u}_h,p_h)\in\bm{V}_h\times P_h$.
Moreover,
\begin{equation}
\label{e: abstract error}
\begin{aligned}
|\bm{u}-\bm{u}_h|_{1,h}&\leq C\left(\inf_{\bm{v_h}\in\bm{Z}_h}|\bm{u}-\bm{v}_h|_{1,h}
+\sup_{\bm{w}_h\in\bm{V}_h}\frac{E_{1,h}(\bm{u},p,\bm{w}_h)}{|\bm{w}_h|_{1,h}}\right),\\
\|p-p_h\|_0&\leq \|p-\mathcal{P}_hp\|_0+C\left(\inf_{\bm{v_h}\in\bm{Z}_h}|\bm{u}-\bm{v}_h|_{1,h}
+\sup_{\bm{w}_h\in\bm{V}_h}\frac{E_{1,h}(\bm{u},p,\bm{w}_h)}{|\bm{w}_h|_{1,h}}\right),
\end{aligned}
\end{equation}
where
\[
\bm{Z}_h=\{\bm{v}_h\in\bm{V}_h:~b_h(\bm{v}_h,q_h)=0,~\forall q_h\in P_h\}
=\{\bm{v}_h\in\bm{V}_h:~\mathrm{div}_h\bm{v}_h=0\},
\]
and the consistency term $E_{1,h}(\bm{u},p,\bm{w}_h)$ is given by
\begin{equation}
\label{e: consistency error vector}
\begin{aligned}
E_{1,h}(\bm{u},p,\bm{w}_h)&=b_h(\bm{w}_h,p)-a_h(\bm{u},\bm{w}_h)+(\bm{f},\bm{w}_h)\\
&=\sum_{K\in\mathcal{T}_h}\left(-\int_{\partial K}\frac{\partial\bm{u}}{\partial\bm{n}}\cdot
\bm{w}_h\,\mathrm{d}s+\int_{\partial K}p\bm{w}_h\cdot\bm{n}\,\mathrm{d}s\right).
\end{aligned}
\end{equation}
It follows from (\ref{e: interp err vector}) that
\[
\inf_{\bm{v_h}\in\bm{Z}_h}|\bm{u}-\bm{v}_h|_{1,h}\leq |\bm{u}-\bm{\Pi}_h\bm{u}|_{1,h}\leq Ch^2|\phi|_4
\]
as $\bm{\Pi}_h\bm{u}\in\bm{Z}_h$.
We shall therefore consider the consistency error.
To this end, we rewrite $E_{1,h}(\bm{u},p,\bm{w}_h)$ as
\[
\begin{aligned}
E_{1,h}(\bm{u},p,\bm{w}_h)=&-\sum_{K\in\mathcal{T}_h}
\int_{\partial K}\frac{\partial\bm{u}}{\partial\bm{n}}\cdot
(\bm{w}_h\cdot\bm{n})\bm{n}\,\mathrm{d}s
-\sum_{K\in\mathcal{T}_h}\int_{\partial K}\frac{\partial\bm{u}}{\partial\bm{n}}\cdot
(\bm{w}_h\cdot\bm{t})\bm{t}\,\mathrm{d}s\\
&+\sum_{K\in\mathcal{T}_h}\int_{\partial K}p\bm{w}_h\cdot\bm{n}\,\mathrm{d}s\\
:=&-J_1(\bm{u},\bm{w}_h)-J_2(\bm{u},\bm{w}_h)+J_3(p,\bm{w}_h)
\end{aligned}
\]
In fact, the weak continuity of $\bm{w}_h$ implies
\begin{equation}
\label{e: normal weak stokes}
\int_E\left[\bm{w}_h\cdot\bm{n}_E\right]_E\xi_E\,\mathrm{d}s=0,~\forall \xi_E\in P_1(E),
~\forall E\in\mathcal{E}_h,~\forall \bm{w}_h\in \bm{V}_h.
\end{equation}
Using a similar strategy as in (\ref{e: I4}) one sees
\begin{equation}
\label{e: 1 order otho vector est}
\begin{aligned}
|J_1(\bm{u},\bm{w}_h)|&
\left\{
\begin{aligned}
&\leq Ch^2|\bm{u}|_3|\bm{w}_h|_{1,h}\leq Ch^2|\phi|_4|\bm{w}_h|_{1,h};\\
&\leq Ch^3|\bm{u}|_3|\bm{w}_h|_{2,h}\leq Ch^3|\phi|_4|\bm{w}_h|_{2,h},\\
\end{aligned}
\right.\\
|J_3(p,\bm{w}_h)|&
\left\{
\begin{aligned}
&\leq Ch^2|p|_2|\bm{w}_h|_{1,h};\\
&\leq Ch^3|p|_2|\bm{w}_h|_{2,h},
\end{aligned}
\right.
\end{aligned}
\end{equation}
thus it remains to estimate $J_2(\bm{u},\bm{w}_h)$.

If we write $\bm{u}=(u_1,u_2)^T$ and $\bm{w}_h=(w_{h1},w_{h2})^T$,
notice from the weak continuity of $\bm{w}_h$ that
\begin{equation}
\label{e: J2}
\begin{aligned}
J_2(\bm{u},\bm{w}_h)=&\sum_{K\in\mathcal{T}_h}\left(
\int_{E_{3,K}}\frac{\partial u_1}{\partial y}
\left(w_{h1}-\mathcal{P}_0^{E_{3,K}}w_{h1}\right)\,\mathrm{d}s
-\int_{E_{1,K}}\frac{\partial u_1}{\partial y}
\left(w_{h1}-\mathcal{P}_0^{E_{1,K}}w_{h1}\right)\,\mathrm{d}s\right)\\
&+\sum_{K\in\mathcal{T}_h}\left(
\int_{E_{2,K}}\frac{\partial u_2}{\partial x}
\left(w_{h2}-\mathcal{P}_0^{E_{2,K}}w_{h2}\right)\,\mathrm{d}s
-\int_{E_{4,K}}\frac{\partial u_2}{\partial x}
\left(w_{h2}-\mathcal{P}_0^{E_{4,K}}w_{h2}\right)\,\mathrm{d}s\right)\\
:=&\sum_{K\in\mathcal{T}_h}J_{1,K}(u_1,w_{h1})+\sum_{K\in\mathcal{T}_h}J_{2,K}(u_2,w_{h2}).
\end{aligned}
\end{equation}
Let us again focus only on $J_{1,K}(u_1,w_{h1})$.
On one hand,
\begin{equation}
\label{e: bilinear vector est 1}
|J_{1,K}(u_1,w_{h1})|\leq Ch\left|\frac{\partial u_1}{\partial y}\right|_{1,K}|w_{h1}|_{1,K}.
\end{equation}
On the other hand, define
\[
J'_{1,K}(u_1,w_{h1})=\frac{h_x^2}{12}\int_K\frac{\partial^2u_1}{\partial x\partial y}\frac{\partial^2w_{h1}}{\partial x\partial y}+\frac{\partial^3u_1}{\partial x^2\partial y}\frac{\partial w_{h1}}{\partial y}\,\mathrm{d}x\mathrm{d}y,
\]
then
\begin{equation}
\label{e: bilinear vector est 2}
|J'_{1,K}(u_1,w_{h1})|\leq Ch\left(\left|\frac{\partial u_1}{\partial y}\right|_{1,K}+h\left|\frac{\partial u_1}{\partial y}\right|_{2,K}\right)|w_{h1}|_{1,K}.
\end{equation}
Moreover, it follows from the definition of $\bm{V}_K$ that $w_{h1}|_K\in P_2(K)\oplus\mathrm{span}\{y^3\}$,
and hence,
\[
w_{h1}|_{E_{i,K}}\in P_2(E_{i,K}),~i=1,3.
\]
The Simpson quadrature rule ensures
\begin{equation}
\label{e: wh1 relation vector}
\frac{1}{h_x}\int_{E_{i,K}}w_{h1}\xi_{i,K}\,\mathrm{d}s=\frac{1}{6}\left(w_{h1}(V_{i'',K})-w_{h1}(V_{i',K})\right),
\end{equation}
where $i''=3$ and $i'=4$ for $i=3$, and $i''=2$ and $i'=1$ for $i=1$.
An argument just like in (\ref{e: normal relation scalar}) and (\ref{e: derivation x scalar}) leads to
\begin{equation}
\label{e: bilinear vector est 3}
J_{1,K}(u_1,w_{h1})-J'_{1,K}(u_1,w_{h1})=0,~\forall \frac{\partial u_1}{\partial y}\in P_1(K),
~\forall w_{h1}\in P_2(K)\oplus\mathrm{span}\{y^3\}.
\end{equation}
Collecting (\ref{e: bilinear vector est 1}), (\ref{e: bilinear vector est 2}) and (\ref{e: bilinear vector est 3})
with the Bramble-Hilbert lemma gives
\begin{equation}
\label{e: bilinear vector est 4}
|J_{1,K}(u_1,w_{h1})-J'_{1,K}(u_1,w_{h1})|\leq Ch^2\left|\frac{\partial u_1}{\partial y}\right|_{2,K}|w_{h1}|_{1,K}
\leq Ch^2|\bm{u}|_{3,K}|\bm{w}_h|_{1,K}.
\end{equation}
Substituting (\ref{e: bilinear vector est 4}) into (\ref{e: J2}) results in
\begin{equation}
\label{e: J1K}
\sum_{K\in\mathcal{T}_h}J_{1,K}(u_1,w_{h1})\leq \sum_{K\in\mathcal{T}_h}J'_{1,K}(u_1,w_{h1})+Ch^2|\phi|_4|\bm{w}_h|_{1,h}.
\end{equation}

Let us turn to the right-hand side in (\ref{e: J1K}).
We need to define the following interpolation operator $\Pi_{E_{i,K}}$ from $P_2(K)\oplus\mathrm{span}\{y^3\}$ to $P_1(E_{i,K})$
by setting
\[
\begin{aligned}
\int_{E_{i,K}}\Pi_{E_{i,K}}v\,\mathrm{d}s&=\int_{E_{i,K}}v\,\mathrm{d}s,\\
\int_{E_{i,K}}\left(\Pi_{E_{i,K}}v\right)\xi_{i,K}\,\mathrm{d}s&=\int_{E_{i,K}}v\xi_{i,K}\,\mathrm{d}s,\\
\forall v\in P_2(K)\oplus\mathrm{span}&\{y^3\},~i=2,4.
\end{aligned}
\]
A simple calculation shows that this interpolation is well-defined.
Moreover,
\begin{equation}
\label{e: key vector 2}
v|_{E_{2,K}}-\Pi_{E_{2,K}}v=v|_{E_{4,K}}-\Pi_{E_{4,K}}v,~\forall v\in P_2(K)\oplus\mathrm{span}\{y^3\}.
\end{equation}
Hence, according to the weak continuity of the normal components of $\bm{w}_h$ over $E_{2,K}$ and $E_{4,K}$,
it holds that
\[
\begin{aligned}
\sum_{K\in\mathcal{T}_h}J'_{1,K}(u_1,w_{h1})
&=\frac{h_x^2}{12}\sum_{K\in\mathcal{T}_h}\int_{y'_K}^{y''_K}\frac{\partial^2u_1}{\partial x\partial y}\frac{\partial w_{h1}}{\partial y}\bigg|_{x=x''_K}-\frac{\partial^2u_1}{\partial x\partial y}\frac{\partial w_{h1}}{\partial y}\bigg|_{x=x'_K}\,\mathrm{d}y\\
&=\frac{h_x^2}{12}\sum_{K\in\mathcal{T}_h}\Bigg(\int_{E_{2,K}}\left(\frac{\partial^2u_1}{\partial x\partial y}
-\mathcal{P}_0^K\frac{\partial^2u_1}{\partial x\partial y}\right)
\left(\frac{\partial w_{h1}}{\partial y}-\frac{\partial\Pi_{E_{2,K}}w_{h1}}{\partial y}\right)\,\mathrm{d}s\\
&~~~~~~~~~~~~~~~~~~-\int_{E_{4,K}}\left(\frac{\partial^2u_1}{\partial x\partial y}
-\mathcal{P}_0^K\frac{\partial^2u_1}{\partial x\partial y}\right)
\left(\frac{\partial w_{h1}}{\partial y}-\frac{\partial\Pi_{E_{4,K}}w_{h1}}{\partial y}\right)\,\mathrm{d}s\Bigg).
\end{aligned}
\]
Hence,
\begin{equation}
\label{e: tangential estimate stokes}
\left|\sum_{K\in\mathcal{T}_h}J'_{1,K}(u_1,w_{h1})\right|
\left\{
\begin{aligned}
&\leq Ch^2\sum_{K\in\mathcal{T}_h}Ch^{1/2}|u_1|_{3,K}Ch^{-1}Ch^{1/2}|w_{h1}|_{1,K}
\leq Ch^2|\bm{u}|_3|\bm{w}_h|_{1,h},\\
&\leq Ch^2\sum_{K\in\mathcal{T}_h}Ch^{1/2}|u_1|_{3,K}Ch^{-1}Ch^{3/2}|w_{h1}|_{2,K}
\leq Ch^3|\bm{u}|_3|\bm{w}_h|_{2,h},
\end{aligned}
\right.
\end{equation}
and therefore by (\ref{e: J1K}),
\[
\left|\sum_{K\in\mathcal{T}_h}J_{1,K}(u_1,w_{h1})\right|\leq Ch^2|\phi|_4|\bm{w}_h|_{1,h}.
\]
An almost identical argument applies to $J_{2,K}(u_2,w_{h2})$ along with (\ref{e: J2}) gives
\[
|J_2(\bm{u},\bm{w}_h)|\leq Ch^2|\phi|_4|\bm{w}_h|_{1,h},
\]
and then
\[
|E_{1,h}(\bm{u},p,\bm{w}_h)|\leq Ch^2(|\phi|_4+|p|_2)|\bm{w}_h|_{1,h},
\]
which establishes the error estimate for $|\bm{u}-\bm{u}_h|_{1,h}$ due to (\ref{e: abstract error}).
Again owing to (\ref{e: abstract error}), the estimate for the pressure is simply based on the fact above and the approximation order $O(h)$ of the projection $\mathcal{P}_h$.
\end{proof}

From Theorem \ref{th: converge stokes}, the velocity error has a higher accuracy of order $O(h^2)$,
while the pressure can only achieve an $O(h)$ convergence rate,
as the approximation error is only $O(h)$.
Nevertheless, this can be improved by a simple postprocessing
based on the idea from \cite{Falk2013} or \cite{Neilan2018}.
For any $K\in\mathcal{T}_h$, we define $p_K^*\in P_1(K)$ by setting
\[
\begin{aligned}
(\nabla p^*_K,\nabla q)_K &= (\Delta\bm{u}_h+\bm{f},\nabla q)_K,~\forall q \in P_1(K),\\
\int_K p^*_K \,\mathrm{d}x\mathrm{d}y&=\int_K p_h\,\mathrm{d}x\mathrm{d}y.
\end{aligned}
\]
Then the postprocessed solution $p_h^*$ is discontinuous and piecewise linear,
defined via $p_h^*|_K=p_K^*$, $\forall K\in\mathcal{T}_h$.
The following result can be derived in a similar manner as Theorem 4.4 in \cite{Neilan2018},
and so the proof is omitted.

\begin{theorem}
\label{th: post p}
Under the same assumptions in Theorem \ref{th: converge stokes}, we have
\[
\|p-p_h^*\|_0\leq Ch^2(|\phi|_4+|p|_2).
\]
\end{theorem}

In what follows, we shall provide the $L^2$-error estimate for the velocity by a duality argument.
To begin with, an $H(\mathrm{div})$-conforming but $H^1$-nonconforming element will be constructed using $W_h^*$.
Over each $K\in\mathcal{T}_h$, the shape function space $\bm{V}_K^*$ is defined by
\[
\begin{aligned}
\bm{V}_K^*&=[P_1(K)]^2+\mathrm{span}\{\curl\,w,~w\in W_K^*\}\\
&=\mathrm{span}\left\{(x,y)^T\right\}\oplus\mathrm{span}\{\curl\,w,~w\in W_K^*\},
\end{aligned}
\]
and the DoF set is given as
\[
\Sigma_K^*=\left\{\bm{v}(V_{i,K}),\int_{E_{i,K}}\bm{v}\,\mathrm{d}s,~i=1,2,3,4\right\},~\forall \bm{v}\in \bm{V}_K^*.
\]
Using the technique provided in Theorem 2.5 in \cite{Zhou2018} and the unisolvency of $W_K^*$ with respect to $T_K^*$,
the element $(K,\bm{V}_K^*,\Sigma_K^*)$ is well-defined.
By the Simpson quadrature rule applied to $(x,y)^T$ and (\ref{e: linear relations scalar}),
the following relations hold:
\begin{equation}
\label{e: linear relations Stokes}
\begin{aligned}
\frac{1}{h_x}\int_{E_{i,K}}v_1\xi_{i,K}\,\mathrm{d}s
&=\frac{1}{6}\left(v_1(V_{i'',K})-v_1(V_{i',K})\right),~i=1,3,\\
\frac{1}{h_y}\int_{E_{j,K}}v_2\xi_{j,K}\,\mathrm{d}s
&=\frac{1}{6}\left(v_2(V_{j'',K})-v_2(V_{j',K})\right),~j=2,4,
~\forall \bm{v}=(v_1,v_2)^T\in \bm{V}_K^*
\end{aligned}
\end{equation}
with $i''=3$ and $i'=4$ for $i=3$, and $i''=2$ and $i'=1$ for $i=1$;
$j''=3$ and $j'=2$ for $j=2$, and $j''=4$ and $j'=1$ for $j=4$.
The global finite element space $\bm{V}_h^*$ is set as
\[
\begin{aligned}
\bm{V}_h^*=\Bigg\{\bm{v}&\in [L^2(\Omega)]^2:~\bm{v}|_K\in \bm{V}_K^*,~\forall K\in\mathcal{T}_h,
~\mbox{$\bm{v}$ is continuous at all}\\
&\mbox{$V\in\mathcal{V}_h^i$ and vanishes at all $V\in\mathcal{V}_h^b$},
~\int_E\left[\bm{v}\right]_E\,\mathrm{d}s=\bm{0}
~\mbox{for all $E\in\mathcal{E}_h$}\Bigg\}.
\end{aligned}
\]
Clearly, $\bm{V}_h^*\in \bm{H}_0(\mathrm{div};\Omega)$.
For a fixed $K$, the interpolation $\bm{\Pi}_K^*$ from $[H^2(K)]^2$ to $\bm{V}_K^*$ is defined such that $\sigma(\bm{\Pi}_K^*\bm{v})=\sigma(\bm{v})$,
$\forall \bm{v}\in [H^2(K)]^2$, $\forall \sigma\in \Sigma_K^*$.
The global version $\bm{\Pi}_h^*$ from $[H^2(\Omega)\cap H_0^1(\Omega)]^2$ to $\bm{V}_h^*$ is naturally determined by $\bm{\Pi}_h^*|_K=\bm{\Pi}_K^*$, $\forall K\in\mathcal{T}_h$.

\begin{remark}
Using a standard argument (see e.g.~Theorem 4.1 in \cite{Zhou2018}),
we have constructed the following discrete Stokes complex:
\[
\begin{tikzcd}[column sep=large, row sep=large]
0 \arrow{r} & W_h^* \arrow{r}{\curl_h}
& \bm{V}_h^* \arrow{r}{\mathrm{div}_h}& P_h\arrow{r}&0.
\end{tikzcd}
\]
\end{remark}

\begin{theorem}
\label{th: converge stokes L2}
Let the solution domain $\Omega$ be convex.
In addition, under the same assumptions in Theorem \ref{th: converge stokes}, we have
\[
\|\bm{u}-\bm{u}_h\|_0\leq Ch^3(|\phi|_4+|p|_2).
\]
\end{theorem}
\begin{proof}
Set $\bm{e}_h=\bm{u}-\bm{u}_h$ and consider the following dual problem:
Find $\bm{\psi}$ and $\chi$ satisfying
\begin{equation}
\label{e: stokes dual problem}
\begin{aligned}
-\Delta\bm{\psi}+\nabla \chi&=\bm{e}_h~~~~~~&\mbox{in}~\Omega,\\
\mbox{div}\,\bm{\psi}&=0~&\mbox{in}~\Omega,\\
\bm{\psi}&=\bm{0}~&\mbox{on}~\partial\Omega,
\end{aligned}
\end{equation}
with the weak form
\[
\begin{aligned}
(\nabla\bm{\psi},\nabla\bm{v})-(\mathrm{div}\,\bm{v},\chi)&=(\bm{e}_h,\bm{v}),~&\forall \bm{v}\in[H_0^1(\Omega)]^2,\\
(\mathrm{div}\,\bm{\psi},q)&=0,~&\forall q\in L_0^2(\Omega),
\end{aligned}
\]
where $(\bm{\psi},\chi)\in [H_0^1(\Omega)]^2\times L_0^2(\Omega)$.
Owing to the convexity of $\Omega$, it follows from \cite{Grisvard1985} that
$(\bm{\psi},\chi)\in [H^2(\Omega)]^2\times H^1(\Omega)$ yielding the regularity
\begin{equation}
\label{e: dual regular stokes}
|\bm{\psi}|_2+|\chi|_1\leq C\|\bm{e}_h\|_0.
\end{equation}
Multiplying $\bm{e}_h$ on both sides of (\ref{e: stokes dual problem}) and integrating by parts show that
\begin{equation}
\label{e: dual test e_h}
\begin{aligned}
\|\bm{e}_h\|_0^2&=a_h(\bm{\psi},\bm{e}_h)-b_h(\bm{e}_h,\chi)+E_{2,h}(\bm{\psi},\chi,\bm{e}_h)\\
&=a_h(\bm{\psi},\bm{e}_h)+E_{2,h}(\bm{\psi},\chi,\bm{e}_h).
\end{aligned}
\end{equation}
On the other hand,
we define $\bm{\psi}^*=\bm{\Pi}_h^*\bm{\psi}\in\bm{V}_h^*$ and then $\bm{\Pi}_h\bm{\psi}^*\in\bm{V}_h$ in light of the definitions of $\bm{V}_h^*$ and $\bm{V}_h$ and the fact that $\bm{\psi}^*\in\bm{H}_0(\mathrm{div};\Omega)$.
Moreover, on each $K$, by Green's formula,
\[
\begin{aligned}
\mathrm{div}\,(\bm{\Pi}_h\bm{\psi}^*|_K)&=\frac{1}{|K|}\int_K\mathrm{div}\,\bm{\Pi}_h\bm{\psi}^*\,\mathrm{d}x\mathrm{d}y
=\frac{1}{|K|}\int_K\mathrm{div}\,\bm{\psi}^*\,\mathrm{d}x\mathrm{d}y\\
&=\frac{1}{|K|}\int_K\mathrm{div}\,\bm{\psi}\,\mathrm{d}x\mathrm{d}y=0,
\end{aligned}
\]
therefore $\bm{\Pi}_h\bm{\psi}^*\in\bm{Z}_h$.
Take $\bm{v}_h=\bm{\Pi}_h\bm{\psi}^*$ in (\ref{e: stokes discrete}), multiply $\bm{\Pi}_h\bm{\psi}^*$
on both sides of (\ref{e: stokes problem}) and integrate by parts, then
\begin{equation}
\label{e: test pihe_h^*}
a_h(\bm{e}_h,\bm{\Pi}_h\bm{\psi}^*)+E_{1,h}(\bm{u},p,\bm{\Pi}_h\bm{\psi}^*)=0.
\end{equation}
The difference of (\ref{e: dual test e_h}) and (\ref{e: test pihe_h^*}) gives
\begin{equation}
\label{e: basic dual stokes}
\|\bm{e}_h\|_0^2=a_h(\bm{\psi}-\bm{\Pi}_h\bm{\psi}^*,\bm{e}_h)+E_{1,h}(\bm{\psi},\chi,\bm{e}_h)
-E_{1,h}(\bm{u},p,\bm{\Pi}_h\bm{\psi}^*).
\end{equation}
The first term is bounded by the triangle inequality, the Bramble-Hilbert lemma and (\ref{e: converge stokes}):
\begin{equation}
\label{e: final dual 1}
\begin{aligned}
|a_h(\bm{\psi}-\bm{\Pi}_h\bm{\psi}^*,\bm{e}_h)|&\leq |\bm{\psi}-\bm{\Pi}_h\bm{\psi}^*|_{1,h}|\bm{e}_h|_{1,h}\\
&\leq \left(|\bm{\psi}-\bm{\Pi}_h^*\bm{\psi}|_{1,h}+|\bm{\psi}^*-\bm{\Pi}_h\bm{\psi}^*|_{1,h}\right)|\bm{e}_h|_{1,h}\\
&\leq Ch^3(|\phi|_4+|p|_2)|\bm{\psi}|_2.
\end{aligned}
\end{equation}
The second term is estimated using the lowest order inter-element orthogonality
(\ref{e: tangential weak stokes}) and (\ref{e: normal weak stokes}):
\begin{equation}
\label{e: final dual 2}
|E_{1,h}(\bm{\psi},\chi,\bm{e}_h)|\leq Ch(|\bm{\psi}|_2+|\chi|_1)|\bm{e}_h|_{1,h}
\leq Ch^3(|\bm{\psi}|_2+|\chi|_1)(|\phi|_4+|p|_2).
\end{equation}
As far as the last term is concerned, we see from (\ref{e: 1 order otho vector est}) that
\begin{equation}
\label{e: final dual 3}
\begin{aligned}
|E_{1,h}(\bm{u},p,\bm{\Pi}_h\bm{\psi}^*)|&\leq |J_1(\bm{u},\bm{\Pi}_h\bm{\psi}^*)|+|J_2(\bm{u},\bm{\Pi}_h\bm{\psi}^*)|+|J_3(p,\bm{\Pi}_h\bm{\psi}^*)|\\
&\leq Ch^3(|\phi|_4+|p|_2)|\bm{\Pi}_h\bm{\psi}^*|_{2,h}+|J_2(\bm{u},\bm{\Pi}_h\bm{\psi}^*)|\\
&\leq Ch^3(|\phi|_4+|p|_2)|\bm{\psi}|_2+|J_2(\bm{u},\bm{\Pi}_h\bm{\psi}^*)|.
\end{aligned}
\end{equation}
Hence, it suffices to investigate $J_2(\bm{u},\bm{\Pi}_h\bm{\psi}^*)$.

To proceed, we decompose $J_2(\bm{u},\bm{\Pi}_h\bm{\psi}^*)$ by
\begin{equation}
\label{e: J2 dual}
J_2(\bm{u},\bm{\Pi}_h\bm{\psi}^*)=J_2(\bm{u},\bm{\Pi}_h\bm{\psi}^*-\bm{\psi}^*)+J_2(\bm{u},\bm{\psi}^*).
\end{equation}
It follows from (\ref{e: linear relations Stokes}) and the continuity of $\bm{\psi}^*$ at the two endpoints of each edge that
\[
\int_E\left[\bm{\psi}^*\cdot\bm{t}_E\right]_E\xi_E\,\mathrm{d}s=0,~\forall \xi_E\in P_1(E),~\forall E\in\mathcal{E}_h,
\]
therefore
\begin{equation}
\label{e: J2 dual part 2}
|J_2(\bm{u},\bm{\psi}^*)|\leq\sum_{K\in\mathcal{T}_h}Ch^{3/2}|\bm{u}|_{3,K}Ch^{3/2}|\bm{\psi}^*|_{2,K}
\leq Ch^3|\phi|_4|\bm{\psi}|_2.
\end{equation}
If we write $\bm{\Pi}_h\bm{\psi}^*-\bm{\psi}^*=(\psi_{h1}^-,\psi_{h2}^-)^T$ for short,
then
\begin{equation}
\label{e: J2 dual new}
J_2(\bm{u},\bm{\Pi}_h\bm{\psi}^*-\bm{\psi}^*)=\sum_{K\in\mathcal{T}_h}J_{1,K}(u_1,\psi_{h1}^-)
+\sum_{K\in\mathcal{T}_h}J_{2,K}(u_2,\psi_{h2}^-).
\end{equation}
From the first relation in (\ref{e: linear relations Stokes}), (\ref{e: bilinear vector est 3})
and the derivation of (\ref{e: bilinear vector est 3}),
we have
\[
J_{1,K}(u_1,\psi_{h1}^-)-J'_{1,K}(u_1,\psi_{h1}^-)=0,~\forall \frac{\partial u_1}{\partial y}\in P_1(K),
\]
and thus by the Bramble-Hilbert lemma as in (\ref{e: bilinear vector est 4}),
\begin{equation}
\label{e: normal to tangent dual stokes}
\begin{aligned}
\left|J_{1,K}(u_1,\psi_{h1}^-)-J'_{1,K}(u_1,\psi_{h1}^-)\right|
&\leq Ch^2\left|u_1\right|_{3,K}|\psi_{h1}^-|_{1,K}\\
&\leq Ch^3|\bm{u}|_{3,K}|\bm{\Pi}_h\bm{\psi}^*-\bm{\psi}^*|_{1,K}\\
&\leq Ch^3|\bm{u}|_{3,K}|\bm{\psi}^*|_{2,K}\leq Ch^3|\phi|_{4,K}|\bm{\psi}|_{2,K}.
\end{aligned}
\end{equation}
Substituting (\ref{e: normal to tangent dual stokes}) into (\ref{e: J2 dual}) results in
\begin{equation}
\label{e: J1K dual}
\sum_{K\in\mathcal{T}_h}J_{1,K}(u_1,\psi_{h1}^-)
\leq \sum_{K\in\mathcal{T}_h}J'_{1,K}(u_1,\psi_{h1}^-)+Ch^3|\phi|_4|\bm{\psi}|_2.
\end{equation}
Moreover, since $\bm{\psi}^*\in\bm{H}_0(\mathrm{div};\Omega)$,
we find by considering (\ref{e: tangential estimate stokes}) that
\[
\left|\sum_{K\in\mathcal{T}_h}J'_{1,K}(u_1,\psi_{h1}^-)\right|
\leq Ch^3|\bm{u}|_3|\bm{\Pi}_h\bm{\psi}^*|_{2,h}\leq Ch^3|\phi|_4|\bm{\psi}|_2,
\]
and therefore by (\ref{e: J1K dual}),
\[
\left|\sum_{K\in\mathcal{T}_h}J_{1,K}(u_1,\psi_{h1}^-)\right|\leq Ch^3|\phi|_4|\bm{\psi}|_2.
\]
The analysis for $J_{2,K}(u_2,\psi_{h2}^-)$ is similar, which along with (\ref{e: J2 dual}), (\ref{e: J2 dual part 2})
(\ref{e: J2 dual new}) gives
\[
|J_2(\bm{u},\bm{\Pi}_h\bm{\psi}^*)|\leq Ch^3|\phi|_4|\bm{\psi}|_2.
\]
Collecting this result and (\ref{e: final dual 1}), (\ref{e: final dual 2}), (\ref{e: final dual 1}) with
(\ref{e: basic dual stokes}), one gets
\[
\|\bm{e}_h\|_0^2\leq Ch^3(|\phi|_4+|p|_2)(|\bm{\psi}|_2+|\chi|_1)\leq Ch^3(|\phi|_4+|p|_2)\|\bm{e}_h\|_0,
\]
where we have used the regularity (\ref{e: dual regular stokes}).
Finally, this theorem is established by dividing both sides by $\|\bm{e}_h\|_0$.
\end{proof}

\begin{remark}
Again, the mixed finite element designed for Brinkman problems in \cite{Zhou2018} has similar properties
as the pair $\bm{V}_h\times P_h$ in this work,
and so the arguments in Theorems \ref{th: converge stokes} and \ref{th: converge stokes L2} are also valid.
This explains the high accuracy phenomenon observed in the numerical tests
in \cite{Zhou2018} under uniform rectangular partitions.
\end{remark}

\section{Numerical examples}
\label{s: numer ex}

Numerical tests are given in this section.
The solution domain is set as $[0,2]\times[0,1]$,
with uniform $n\times n$ rectangular partitions $\{\mathcal{T}_h\}$ constructed.
It is clear that
\[
h_x=\frac{2}{n},~h_y=\frac{1}{n},~h^2=h_x^2+h_y^2.
\]

We first test $W_h$ for the biharmonic problem (\ref{e: model problem scalar}),
where the exact solution is given by
\begin{equation}
\label{e: example scalar}
u =(3x^2-2y+6xy^2)(x(x-2)y(y-1))^2.
\end{equation}
The errors in various norms are illustrated in Table \ref{t: scalar ours}.
One sees that the convergence orders in discrete $H^2$- and $H^1$-norms are precisely $O(h^2)$ and $O(h^3)$,
respectively, as predicted in Theorems \ref{th: converge scalar} and \ref{th: converge scalar H1}.
The $L^2$ error seems to be of order four, which will be studied in our future work.

\begin{table}[!htb]
\begin{center}
\begin{tabular}{p{1cm}<{\centering}p{2cm}<{\centering}p{1cm}<{\centering}p{2cm}<{\centering}p{1cm}<{\centering}
p{2cm}<{\centering}p{1.5cm}<{\centering}}
\toprule
$n$      &$|u-u_h|_{2,h}$ & order & $|u-u_h|_{1,h}$ & order &$\|u-u_h\|_0$ & order  \\
\midrule
$4$      &1.209E0     &       &7.041E-2   &               &9.209E-3     &       \\
$8$      &3.528E-1    &1.78   &9.173E-3   &2.94           &4.139E-4     &4.48   \\
$16$     &8.880E-2    &1.99   &1.063E-3   &3.11           &2.060E-5     &4.33   \\
$32$     &2.140E-2    &2.05   &1.242E-4   &3.10           &1.202E-6     &4.10   \\
$64$     &5.198E-3    &2.04   &1.487E-5   &3.06           &7.428E-8     &4.02   \\
\bottomrule
\end{tabular}
\caption{The discrete $H^2$, $H^1$ and $L^2$ errors produced by $W_h$ applied to the biharmonic problem determined by (\ref{e: example scalar}) for different $n$.
 \label{t: scalar ours}}
\end{center}
\end{table}

As a comparison, we also check the performance of the well-known Adini element \cite{Lascaux1975} for the same problem.
The numbers of local DoFs are both 12 for both elements,
and the global DoFs of the Adini element are fewer than those of $W_h$.
However, the space $W_h$ is highly nonconforming,
enjoying a cheap local communication when the method is implemented,
especially for parallel computing.
As the error analysis in \cite{Lascaux1975,Luo2004,Hu2016},
the Adini element can also achieve a second order convergence rate in discrete $H^2$-norm.
Moreover, the absolute errors are sightly lower than the those produced by $W_h$ due to the strong continuity in the tangential direction.
However, the errors in discrete $H^1$- and $L^2$-norms are also only $O(h^2)$,
lower than those of $W_h$,
which is consistent with the lower bound estimate given in \cite{Hu2013,Hu2016}.

\begin{table}[!htb]
\begin{center}
\begin{tabular}{p{1cm}<{\centering}p{2cm}<{\centering}p{1cm}<{\centering}p{2cm}<{\centering}p{1cm}<{\centering}
p{2cm}<{\centering}p{1.5cm}<{\centering}}
\toprule
$n$      &$|u-u_h|_{2,h}$ & order & $|u-u_h|_{1,h}$ & order &$\|u-u_h\|_0$ & order  \\
\midrule
$4$      &1.112E0     &       &1.270E-1   &               &2.195E-2     &       \\
$8$      &3.113E-1    &1.84   &3.060E-2   &2.05           &6.283E-3     &1.80   \\
$16$     &7.681E-2    &2.02   &7.642E-3   &2.00           &1.636E-3     &1.94   \\
$32$     &1.901E-2    &2.01   &1.915E-3   &2.00           &4.137E-4     &1.98   \\
$64$     &4.738E-3    &2.00   &4.791E-4   &2.00           &1.037E-4     &2.00   \\
\bottomrule
\end{tabular}
\caption{The discrete $H^2$, $H^1$ and $L^2$ errors produced by the Adini element for the biharmonic problem determined by (\ref{e: example scalar}) for different $n$.
 \label{t: scalar adini}}
\end{center}
\end{table}

We end this work by testing the divergence-free Stokes element $\bm{V}_h\times P_h$ for the model problem (\ref{e: stokes problem}) determined by
\begin{equation}
\label{e: example stokes}
\bm{u}=\curl\left(\exp(x+2y)(x(x-2)y(y-1))^2\right),~p=-\sin2\pi x\sin2\pi y.
\end{equation}
The numerical results are provided in Table \ref{t: stokes}.
One can observe that the convergence orders for the velocity, the pressure and the postprocessed pressure are achieved,
as predicted in Theorems \ref{th: converge stokes} and  \ref{th: post p}.
Moreover,
the $L^2$ errors of the velocity have an $O(h^3)$ convergence order,
which agrees with the assertion in Theorem \ref{th: converge stokes L2}.

\begin{table}[!htb]
\begin{center}
\begin{tabular}{p{0.7cm}<{\centering}p{1.6cm}<{\centering}p{0.8cm}<{\centering}p{1.6cm}<{\centering}p{0.8cm}<{\centering}
p{1.6cm}<{\centering}p{0.8cm}<{\centering}p{1.6cm}<{\centering}p{0.8cm}<{\centering}}
\toprule
$n$      &$|\bm{u}-\bm{u}_h|_{1,h}$ & order & $\|\bm{u}-\bm{u}_h\|_0$ & order
& $\|p-p_h\|_0$ & order  & $\|p-p_h^*\|_0$ & order\\
\midrule
$4$      &2.473E0     &       &1.100E-1   &               &6.569E-1     &       &1.045E0      &    \\
$8$      &7.505E-1    &1.72   &1.657E-2   &2.73           &3.485E-1     &0.91   &3.187E-1     &1.71\\
$16$     &1.968E-1    &1.93   &2.305E-3   &3.03           &1.772E-1     &0.98   &6.173E-2     &2.37\\
$32$     &4.846E-2    &2.02   &2.399E-4   &3.08           &8.932E-2     &0.99   &9.938E-3     &2.64\\
$64$     &1.188E-2    &2.03   &2.878E-5   &3.06           &4.477E-2     &1.00   &1.869E-3     &2.41\\
\bottomrule
\end{tabular}
\caption{The errors for the velocity and the pressure produced by $\bm{V}_h\times P_h$ for the Stokes problem determined by (\ref{e: example stokes}) for different $n$.
 \label{t: stokes}}
\end{center}
\end{table}

\end{document}